\documentclass[oneside,english]{amsart}
\usepackage[latin9]{inputenc}
\usepackage{mathrsfs}
\usepackage{amstext}
\usepackage{amsthm}
\usepackage{amssymb}

\makeatletter
\numberwithin{equation}{section}
\numberwithin{figure}{section}
\theoremstyle{plain}
\newtheorem{thm}{\protect\theoremname}[section]
\theoremstyle{plain}
\newtheorem{cor}[thm]{\protect\corollaryname}
\theoremstyle{remark}
\newtheorem{rem}[thm]{\protect\remarkname}
\theoremstyle{plain}
\newtheorem{prop}[thm]{\protect\propositionname}
\theoremstyle{plain}
\newtheorem{lem}[thm]{\protect\lemmaname}
\theoremstyle{definition}
\newtheorem{defn}[thm]{\protect\definitionname}

\usepackage[colorlinks=true,linkcolor=blue,urlcolor=black]{hyperref}

\makeatother

\usepackage{babel}
\providecommand{\corollaryname}{Corollary}
\providecommand{\definitionname}{Definition}
\providecommand{\lemmaname}{Lemma}
\providecommand{\propositionname}{Proposition}
\providecommand{\remarkname}{Remark}
\providecommand{\theoremname}{Theorem}

\begin{document}
\global\long\def\AA{\mathbb{A}}%
\global\long\def\BB{\mathbb{B}}%
 
\global\long\def\CC{\mathbb{C}}%
 
\global\long\def\DD{\mathbb{D}}%
 
\global\long\def\EE{\mathbb{E}}%
 
\global\long\def\FF{\mathbb{F}}%
 
\global\long\def\GG{\mathbb{G}}%
 
\global\long\def\HH{\mathbb{H}}%
 
\global\long\def\II{\mathbb{I}}%

\global\long\def\JJ{\mathbb{J}}%
 
\global\long\def\KK{\mathbb{K}}%
 
\global\long\def\LL{\mathbb{L}}%
 
\global\long\def\MM{\mathbb{M}}%
 
\global\long\def\NN{\mathbb{N}}%
 
\global\long\def\OO{\mathbb{O}}%
 
\global\long\def\PP{\mathbb{P}}%
 
\global\long\def\QQ{\mathbb{Q}}%
 
\global\long\def\RR{\mathbb{R}}%
 
\global\long\def\SS{\mathbb{S}}%
 
\global\long\def\TT{\mathbb{T}}%
 
\global\long\def\UU{\mathbb{U}}%
 
\global\long\def\VV{\mathbb{V}}%
 
\global\long\def\WW{\mathbb{W}}%
 
\global\long\def\XX{\mathbb{X}}%
 
\global\long\def\YY{\mathbb{Y}}%
 
\global\long\def\ZZ{\mathbb{Z}}%

\global\long\def\Ac{\mathcal{A}}%
 
\global\long\def\Bc{\mathcal{B}}%
 
\global\long\def\Cc{\mathcal{C}}%
 
\global\long\def\Dc{\mathcal{D}}%
\global\long\def\Ec{\mathcal{E}}%
\global\long\def\Fc{\mathcal{F}}%
 
\global\long\def\Gc{\mathcal{G}}%
\global\long\def\Hc{\mathcal{H}}%
 
\global\long\def\Ic{\mathcal{I}}%
 
\global\long\def\Jc{\mathcal{J}}%
 
\global\long\def\Kc{\text{\ensuremath{\mathcal{K}}}}%
\global\long\def\Lc{\mathcal{L}}%
 
\global\long\def\Mc{\mathcal{M}}%
 
\global\long\def\Nc{\mathcal{N}}%
 
\global\long\def\Oc{\mathcal{O}}%
\global\long\def\Pc{\mathcal{P}}%
 
\global\long\def\Qc{\mathcal{Q}}%
 
\global\long\def\Rc{\mathcal{R}}%
\global\long\def\Sc{\mathcal{S}}%
 
\global\long\def\Tc{\mathcal{T}}%
 
\global\long\def\Uc{\text{\ensuremath{\mathcal{U}}}}%
 
\global\long\def\Vc{\mathcal{V}}%
 
\global\long\def\Wc{\mathcal{W}}%
 
\global\long\def\Xc{\mathcal{X}}%
 
\global\long\def\Yc{\mathcal{Y}}%
\global\long\def\Zc{\mathcal{Z}}%

\global\long\def\Acal{\mathscr{A}}%
\global\long\def\Bcal{\mathscr{B}}%
 
\global\long\def\Ccal{\mathscr{C}}%
 
\global\long\def\Dcal{\mathscr{D}}%
 
\global\long\def\Ecal{\mathscr{E}}%
 
\global\long\def\Fcal{\mathscr{F}}%
\global\long\def\Gcal{\mathscr{G}}%
 
\global\long\def\Hcal{\mathscr{H}}%
 
\global\long\def\Ical{\mathscr{I}}%
 
\global\long\def\Jcal{\mathscr{J}}%
 
\global\long\def\Kcal{\mathscr{K}}%
\global\long\def\Lcal{\mathscr{L}}%
 
\global\long\def\Mcal{\mathscr{M}}%
 
\global\long\def\Ncal{\mathscr{N}}%
 
\global\long\def\Ocal{\mathscr{O}}%
 
\global\long\def\Pcal{\mathscr{P}}%
\global\long\def\Qcal{\mathscr{Q}}%
 
\global\long\def\Rcal{\mathscr{R}}%
 
\global\long\def\Scal{\mathscr{S}}%
 
\global\long\def\Tcal{\mathscr{T}}%
 
\global\long\def\Ucal{\mathscr{U}}%
\global\long\def\Vcal{\mathscr{V}}%
 
\global\long\def\Wcal{\mathscr{W}}%
 
\global\long\def\Xcal{\mathscr{X}}%
 
\global\long\def\Ycal{\mathscr{Y}}%
 
\global\long\def\Zcal{\mathscr{Z}}%

\global\long\def\Ane{\mathbf{A}}%
 
\global\long\def\Bne{\mathbf{B}}%
 
\global\long\def\Cne{\mathbf{C}}%
 
\global\long\def\Dne{\mathbf{D}}%
 
\global\long\def\Ene{\mathbf{E}}%
 
\global\long\def\Fne{\mathbf{F}}%
 
\global\long\def\Gne{\mathbf{G}}%
 
\global\long\def\Hne{\mathbf{H}}%
 
\global\long\def\Ine{\mathbf{I}}%
 
\global\long\def\Jne{\mathbf{J}}%
 
\global\long\def\Kne{\mathbf{K}}%
 
\global\long\def\Lne{\mathbf{L}}%
 
\global\long\def\Nne{\mathbf{N}}%
 
\global\long\def\Ane{\mathbf{A}}%
 
\global\long\def\One{\mathbf{O}}%
 
\global\long\def\Pne{\mathbf{P}}%
 
\global\long\def\Qne{\mathbf{Q}}%
 
\global\long\def\Rne{\mathbf{R}}%
 
\global\long\def\Sne{\mathbf{S}}%
 
\global\long\def\Tne{\mathbf{T}}%
 
\global\long\def\Une{\mathbf{U}}%
 
\global\long\def\Vne{\mathbf{V}}%
 
\global\long\def\Xne{\mathbf{X}}%
 
\global\long\def\Yne{\mathbf{Y}}%
 
\global\long\def\Zne{\mathbf{Z}}%

\global\long\def\Att{\mathtt{A}}%
 
\global\long\def\Btt{\mathtt{B}}%
 
\global\long\def\Ctt{\mathtt{C}}%
 
\global\long\def\Dtt{\mathtt{D}}%
 
\global\long\def\Ett{\mathtt{E}}%
 
\global\long\def\Ftt{\mathtt{F}}%
 
\global\long\def\Gtt{\mathtt{G}}%
 
\global\long\def\Htt{\mathtt{H}}%
 
\global\long\def\Itt{\mathtt{I}}%
 
\global\long\def\Jtt{\mathtt{J}}%
 
\global\long\def\Ktt{\mathtt{K}}%
 
\global\long\def\Ltt{\mathtt{L}}%
 
\global\long\def\Mtt{\mathtt{M}}%
 
\global\long\def\Ntt{\mathtt{N}}%
 
\global\long\def\Ott{\mathtt{O}}%
 
\global\long\def\Ptt{\mathtt{P}}%
 
\global\long\def\Qtt{\mathtt{Q}}%
 
\global\long\def\Rtt{\mathtt{R}}%
 
\global\long\def\Stt{\mathtt{S}}%
 
\global\long\def\Ttt{\mathtt{T}}%
 
\global\long\def\Utt{\mathtt{U}}%
 
\global\long\def\Vtt{\mathtt{V}}%
 
\global\long\def\Wtt{\mathtt{W}}%
 
\global\long\def\Xtt{\mathtt{X}}%
 
\global\long\def\Ytt{\mathtt{Y}}%
 
\global\long\def\Ztt{\mathtt{Z}}%

\global\long\def\Hyp{\mathbb{H}^{2}}%

\global\long\def\SL{\text{\ensuremath{\mathrm{SL}}}}%
\global\long\def\PSL{\mathrm{PSL}}%
\global\long\def\GL{\text{\ensuremath{\mathrm{GL}}}}%
\global\long\def\O{\text{\ensuremath{\mathrm{O}}}}%
\global\long\def\SO{\text{\ensuremath{\mathrm{SO}}}}%
\global\long\def\PSO{\mathrm{PSO}}%
 
\global\long\def\PO{\mathrm{PO}}%

\global\long\def\dd{\,\mathrm{d}}%

\global\long\def\eps{\varepsilon}%

\global\long\def\inv{^{-1}}%
 
\global\long\def\tr{\text{tr\,}}%
\global\long\def\im{\text{Im }}%
 
\global\long\def\re{\text{Re }}%
 
\global\long\def\cl{\text{cl\,}}%

\global\long\def\norm#1{\left\lvert \left\lvert #1\right\rvert \right\rvert }%
 
\global\long\def\trNorm#1{\norm{#1}{}_{\mathrm{tr}}}%
 
\global\long\def\HSnorm#1{\norm{#1}{}_{HS}}%
 
\global\long\def\opnorm#1{\norm{#1}_{op}}%

\global\long\def\scal#1#2{\left\langle #1,#2\right\rangle }%

\global\long\def\gexp{\delta_{\Gamma}}%
 
\global\long\def\llG{\ll_{\Gamma}}%
 
\global\long\def\ggG{\gg_{\Gamma}}%
 
\global\long\def\asymG{\asymp_{\Gamma}}%
 
\global\long\def\len#1{\Upsilon_{#1}}%

\global\long\def\Words{\mathcal{W}}%
\global\long\def\neWords{\mathcal{\mathcal{W}^{\circ}}}%
 
\global\long\def\mirror#1{\widetilde{#1}}%
\global\long\def\Start#1{S(#1)}%
 
\global\long\def\End#1{E(#1)}%
 
\global\long\def\wlength#1{|#1|_{\mathrm{w}}}%
 
\global\long\def\backspace#1{\overleftarrow{#1}}%

\global\long\def\df{\mathrm{def}}%
\global\long\def\eqdf{\stackrel{\df}{=}}%

\title{Spectral gap for random Schottky surfaces}
\author{Irving Calder\'on, Michael Magee and Fr\'ed\'eric Naud}
\begin{abstract}
We establish a spectral gap for resonances of the Laplacian of random
Schottky surfaces which is optimal, according to a conjecture of Jakobson
and Naud.
\end{abstract}

\maketitle

\section{Introduction\label{sec:Intro}}

This paper addresses the question of whether typical hyperbolic Schottky
surfaces have almost optimal spectral gaps. By hyperbolic Schottky
surface we mean an infinite area Riemannian surface $X$ which is
connected, orientable, complete, of constant curvature $-1$, with
finitely generated fundamental group and without cusps. The spectral
gap refers either to a gap in the spectrum of the Laplace-Beltrami
operator $\Delta_{X}$ or, more generally, to a region in the complex
plane where there are no resonances of $X$. 

Let us explain the background on the spectral theory of Schottky surfaces
needed for this work\footnote{Borthwick's book \cite{borthwick_spectral_2016} is a good reference
for the spectral theory of infinite area hyperbolic surfaces.}. One can obtain $X$ as a quotient $\Gamma\backslash\Hyp$ of the
hyperbolic plane $\Hyp$ by a $\Gamma$ discrete, finitely generated,
non-abelian free subgroup of $\PSL(2,\RR)$. We will consider a basis
$\Bc=\{\gamma_{1},\ldots,\gamma_{N}\}$ of $\Gamma$ having the geometric
description presented in Section \ref{subsec:Fuchsian-Schottky-groups}.
Any $\Gamma$-orbit in $\Hyp$ accumulates on a compact subset $\Lambda_{\Gamma}$
of $\partial\Hyp$ called the \emph{limit set }of $\Gamma$. We denote
by $\gexp$ the Hausdorff dimension of $\Lambda_{\Gamma}$. The assumption
that $\Gamma$ is not abelian and $X$ has infinite area implies that
$\gexp\in(0,1)$.

We consider $\Delta_{X}$ as an unbounded operator on $L^{2}(X)$. It is essentially self-adjoint,
positive semidefinite. By results of Lax-Phillips in
\cite{lax_asymptotic_1982}, the continuous spectrum of $\Delta_{X}$
is the interval $\left[\frac{1}{4},\infty\right)$, and the discrete
spectrum consists of finitely many eigenvalues---counted with multiplicity---in
$\left(0,\frac{1}{4}\right)$. A result of Patterson in \cite{patterson_limit_1976}
says that when $\gexp>\frac{1}{2}$, the bottom of the spectrum of
$\Delta_{X}$ is the simple eigenvalue $\lambda_{0}(X)\eqdf\gexp(1-\gexp)$.
In this case, we define the $L^{2}$-spectral gap of $X$ as $\lambda_{1}(X)-\lambda_{0}(X)$,
where $\lambda_{1}(X)$ is the minimum element of the spectrum other
than $\lambda_{0}(X)$. This definition is not well suited for the
case $\gexp\leq\frac{1}{2}$, since the spectrum of $\Delta_{X}$
is precisely $[\frac{1}{4},\infty)$. For this reason, we will work
with the following more general notion of spectral gap, which applies
regardless of the value of $\gexp$: The resolvent operator
\[
R_{X}(s)\eqdf(\Delta_{X}-s(1-s))^{-1}:C_{c}^{\infty}(X)\rightarrow C^{\infty}(X)
\]
is initially defined on the halfplane $\left\{ s\in\CC\mid\re s>\frac{1}{2}\right\} $.
Mazzeo and Melrose show in \cite{mazzeo_meromorphic_1987} that $R_{X}$
has a  meromorphic continuation to the entire complex plane. See also the papers of Guillop\'e and Zworski \cite{GZ1,GZ2}
for the more specific case of surfaces.

The poles
of this family of operators are called \emph{resonances of $X$} and form a discrete subset of $\CC$. Resonances are the natural replacement data to the missing $L^2$
eigenvalues in non-compact situations. They corresponds to non-$L^2$ eigenfunctions satisfying an outgoing condition at infinity, see \cite{borthwick_spectral_2016} for details.
Let $\Rc_{X}$ be the set of resonances of $X$. The multiplicity
of a resonance $s$ is defined by $\mathrm{rank}\left(\int_{\Cc_{s}}R_{X}\right)$,
where $\Cc_{s}$ is a circle enclosing $s$ and no other resonance
of $X$, with anticlockwise orientation. Resonances $s$ with $\re s>\frac{1}{2}$
yield $L^2$-eigenvalues $s(1-s)$ of $\Delta_{X}$ with the same multiplicity.
Hence the gaps in the discrete spectrum of $\Delta_{X}$ can be thought
as regions of $\CC$ without resonances. This is the point of view
we adopt. Here are some important facts about resonances: $\gexp$
is the only resonance with real part $\geq\gexp$ and it is simple.
Moreover, there is always an $\eps>0$ such that
\[
\Rc_{X}\cap\{s\in\CC\mid\re s>\gexp-\eps\}=\{\gexp\}.
\]
This follows from \cite{lax_asymptotic_1982} when $\gexp>\frac{1}{2}$
and from \cite{naud_expanding_2005} when $\gexp\leq\frac{1}{2}$.
We define the \emph{spectral gap of $X$ }as the maximum value $\eps(\Gamma)$
of such $\eps$. On the one hand, it was shown in \cite{jakobson-naud-soares_LargeDegree_2020}
that $\eps(\Gamma)$ can be arbitrarily small. On the other hand,
a conjecture of Jakobson and Naud \cite[Conjecture 2]{jakobson-naud_CriticalLine_2012}
predicts, among other things, that for any $\epsilon>0$ there are
infinitely many resonances of $X$ with $\re s>\frac{\gexp}{2}-\epsilon$,
so $\eps(\Gamma)\leq\frac{\gexp}{2}$. It is a pressing question as
to whether a typical $X$ has spectral gap close to this conjectural
optimal value. This question has famous analogs in graph theory that
discuss below.

We address this question from a probabilistic perspective, using the
model of hyperbolic Schottky surfaces introduced in \cite{magee-naud_explicit_2020}:
consider a Schottky surface $X=\Gamma\backslash\Hyp$ and let $\phi_{n}$
be a random homomorphism $\Gamma\to S_{n}$ with uniform distribution.
We obtain a random diagonal action of $\Gamma$ on $\HH^{2}\times\{1,2\ldots,n\}$,
by setting for all $\gamma \in \Gamma$,
$$\gamma.(z,j):=(\gamma(z),\phi_n(\gamma)(j)). $$
The quotient 
\[
X_{n}\eqdf\Gamma\backslash(\HH^{2}\times\{1,2\ldots,n\})
\]
is then a random $n$-sheeted Riemannian covering of $X$. Note that if $Y$ is a
finite cover of $X$, the value of $\delta$ is the same for both
surfaces and any resonance of $X$ is a resonance of $Y$ with multiplicity
at least as large as that of $X$. Hence $\eps(Y)$ is at most $\eps(X)$.
The main theorem of this paper is the following.
\begin{thm}
\label{thm:sp_gap_random_Schottky}Let $X$ be a nonelementary, hyperbolic
Schottky surface and let $X_{n}$ be a random $n$-sheeted covering
of $X$ with uniform distribution. For any compact $\Kc$ contained
in the halfplane $\{s\in\CC\mid\re s>\gexp/2\}$, the probability
that $X$ and $X_{n}$ have the same resonances in $\Kc$ with the
same multiplicities tends to $1$ as $n\to\infty$.
\end{thm}

This theorem says that a typical finite covering of $X$ has the maximal
possible spectral gap \emph{provided we restrict to resonances of
bounded frequency }(imaginary part)\emph{. }A weaker version of Theorem
\ref{thm:sp_gap_random_Schottky} with $\frac{\gexp}{2}$ replaced
by $\frac{3\gexp}{4}$ was proved in \cite[Thm. 1.1]{magee-naud_explicit_2020}.

Theorem \ref{thm:sp_gap_random_Schottky} together with \cite[Theorem 1.6]{ballmann-matth-mondal_SmallEigenvalues_2017}
implies the following corollary on $L^{2}$ eigenvalues.
\begin{cor}
Let $\Gamma$ be a Fuchsian Schottky group such that $\gexp>\frac{1}{2}$
and the surface\footnote{In this case $X$ is either a pair of pants with three funnels or
a torus with one funnel.} $X=\Gamma\backslash\Hyp$ has Euler characteristic $-1$. Let $X_{n}$
be a uniformly random degree $n$ covering space of $X$. With probability
tending to one as $n\to\infty$, the only eigenvalue of $\Delta_{X_{n}}$
is $\gexp(1-\gexp)$.
\end{cor}

The proof of our spectral gap for random Schottky surfaces relies
on the next result of independent interest. We explain in Section
\ref{subsec:Strateg_of_proof} how these statements are related. See
Section \ref{subsec:Transfer-operators-and} for the definition of
the transfer operators $\Lc_{s,\rho_{j}}$.
\begin{thm}
\label{thm:sc_deterministic}Let $\Gamma$ be a Fuchsian Schottky
group. Suppose that $(\text{\ensuremath{\rho_{j},W_{j})_{j\geq1}}}$
is a sequence of random, finite dimensional unitary representations
of $\Gamma$ that a.a.s. strongly converge to the left regular representation
$(\rho_{\Gamma},\ell^{2}(\Gamma))$ in the sense of Section \ref{subsec:Strong-convergence-and_mat_amp}.
Then, for any compact subset $\Kc$ of the halfplane $\{s\in\CC\mid\re s>\gexp/2\}$
there is an integer $\ell=\ell(\Kc,\Gamma)>0$ such that

\[
\lim_{j\to\infty}\PP\left(\opnorm{\Lc_{s,\rho_{j}}^{\ell}}<1\text{ for all }s\in\Kc\right)=1.
\]
\end{thm}

\begin{rem}
Following the proof of Theorem \ref{thm:sc_deterministic} one can
see that if we assume that $(\rho_{j})_{j\geq1}$ a.s. strongly converge
to $\rho_{\Gamma}$, then there is an $\ell=\ell(\Kc,\Gamma)$ such
that
\[
\sum_{j=1}^{\infty}\left[1-\PP\left(\opnorm{\Lc_{s,\rho_{j}}^{\ell}}<1\text{ for all })s\in\Kc\right)\right]<\infty.
\]
\end{rem}

Theorem \ref{thm:sc_deterministic} shows that the process of `inducing
strong convergence', first uncovered in \cite{hide-magee_NearOptimal}
for finite-area surfaces, is also valid in the infinite volume Schottky
setting in a stronger sense than that of $L^{2}$-eigenvalues.

\subsection{Spectral gaps of random graphs\label{subsec:Spectral-gaps-graphs}}

We now explain the analogy with random graphs---this is also discussed
in detail in the introduction of \cite{magee-naud_explicit_2020}.
We need some notation: If $G$ is a graph with $n$ vertices, we denote
$\lambda_{0}(G)\geq\lambda_{1}(G)\geq\cdots\geq\lambda_{n-1}(G)$
the eigenvalues of its adjacency operator. Note that $\lambda_{0}(G)=d$
when $G$ is $d$-regular. Let $\Gcal_{d,n}$ be a random $d$-regular
graph of size $n$ with uniform distribution. A celebrated conjecture
of Alon \cite{alon_EvExpanders_1986} predicted that for any $\epsilon>0$,
$\lambda_{1}(\Gcal_{d,n})<2\sqrt{d-1}+\epsilon$ with probability
tending to one as $n\to\infty$. The number $2\sqrt{d-1}$ is relevant
for two reasons: it is the spectral radius of the adjacency operator
of the $d$-regular tree---which is the universal cover of any $d$-regular
graph---and also, a result of Alon-Boppana \cite{nili_SecondEv_1991}
says that for any $d$-regular graph on $n$ vertices $G_{d,n}$, 

\[
\lambda_{1}(G_{d,n})\geq2\sqrt{d-1}-o(1)\quad\text{as }n\to\infty.
\]
Hence one cannot replace $2\sqrt{d-1}$ by a smaller number in Alon's
conjeture. This means that, for $d$-regular graphs, $2\sqrt{d-1}$
plays the role of $\frac{\gexp}{2}$ for Schottky surfaces. Even though
the Alon-Boppana bound is not very hard to prove for graphs, we point out that the analog conjecture
of Jakobson-Naud for Schottky surfaces has not been established yet and seems much harder to reach. Alon's
conjecture was proved by Friedman in \cite{friedman_ProofAlonsConj_2008}.
See also \cite{bordenave_NewProof_2020} for a new proof and \cite{puder_Expansion_2015}
for a proof of a slightly weaker result.
Friedman conjectured in \cite{friedman_Relative_2003} a variant of
Alon's conjecture for random covering spaces of degree $n$\footnote{In graph-theoretic literature, these are called $n$-lifts.}
of any initial finite graph $G$ provided that:
\begin{itemize}
\item One replaces $2\sqrt{d-1}$ by the spectral radius $r_{\widetilde{G}}$
of the adjacency operator of the universal cover of $G$;
\item One allows eigenvalues of the adjacency operator that already belonged
to $G$ (as one must).
\end{itemize}
Friedman's conjecture was proved in the breakthrough work \cite{bordenave-collins_EvRandomLifts_2019}
of Bordenave and Collins. In fact, Bordenave-Collins proved a vast
generalization of Friedman's conjecture where one twists a random
Hecke operator, formed from random permutations, by any fixed finite
dimensional matrices, assuming the matrices satisfy a symmetry condition
that forces the resulting operator to be self-adjoint. In fact, the
\cite{bordenave-collins_EvRandomLifts_2019} is a vital ingredient
in the proof of our spectral gap for random Schottky surfaces, as
is explained in Section \ref{subsec:Strateg_of_proof}. 

\subsection{Some related results}

The first result on spectral gap of random hyperbolic surfaces is
due to Brooks and Makover \cite{brooks-makover_Random_2004} who prove,
for a combinatorial model of random closed hyperbolic surfaces depending
on a parameter $n$ that influences the genus (non-deterministically),
that there exists a constant $C>0$ such that $\lambda_{1}\geq C$
with probability tending to one as $n\to\infty$, where $\lambda_{1}$
is the first non-zero eigenvalue of the Laplacian of the surface.
Mirzakhani proved in \cite{mirzakhani_Growth_2013} that for a Weil-Petersson
random closed hyperbolic surface of genus $g$, $\lambda_{1}\geq0.0024$
with probability tending to one as $g\to\infty$. In \cite{magee-naud_explicit_2020},
the second and third named authors of the current work proved that
Theorem \ref{thm:sp_gap_random_Schottky} holds with $\frac{\gexp}{2}$
replaced by $\frac{3\gexp}{4}$.

Returning to closed surfaces, by building on \cite{magee-puder_Asymptotic_2023},
Magee-Naud-Puder proved in \cite{magee-puder-naud_RelSpectralGap_2022}
that for a uniformly random degree $n$ cover $Y_{n}$ of a fixed
closed hyperbolic surface $Y$ and all $\epsilon>0$, the probability
that $Y_{n}$ has no new eigenvalues below $\frac{3}{16}-\epsilon$
tends to $1$ as $n\to\infty$. This result was adapted to Weil-Petersson
random surfaces independently by Wu-Xue \cite{wu-xue_RandHyp_2022}
and Lipnowski-Wright \cite{lipnowski-wright_TowardsOptimal_2024};
here the corresponding statement is that there are no eigenvalues
between $0$ and $\frac{3}{16}-\epsilon$. When it comes \emph{solely}
to $L^{2}$ eigenvalues, these `$\frac{3}{16}$' results are at the
strength of the main result of \cite{magee-naud_explicit_2020} giving
resonance-free regions in terms of $\frac{3\gexp}{4}$: for compact
surfaces, $\gexp=1$ and $\frac{3}{4}\left(1-\frac{3}{4}\right)=\frac{3}{16}$.
On the other hand, closed surfaces involve additional difficulties
due to their non-free fundamental groups.

Some works on Weil-Petersson random surfaces, like \cite{monk_B-SConvergence_2022,gilmore-lemasson-sahlsten-thomas},
do not imply spectral gaps but offer instead spectral delocalization
results .

Uniform spectral gaps for deterministic covering spaces of infinite
area hyperbolic surfaces has also been of interest in number theoretic
settings; see \cite{gamburd_spectral_2002,bourgain-gamburd-sarnak_generalization_2011,oh-winter_uniform_2015,magee-oh-winter_uniform_2019,calderon-magee_explicit_2023}
for a selection of results. The quantitative results here are those
by Gamburd and Calder\'on-Magee. Much of the motivation of these works
came from the `thin groups' research program; see Sarnak's article
\cite{sarnak_NotesThin_2014} for an overview. 

Another closely related concept is that of \emph{essential spectral
gap}. We say that $t\in\RR$ is an essential spectral gap of a hyperbolic
Schottky surface $X=\Gamma\backslash\Hyp$ if the half-plane $\{s\in\CC\mid\re s>t\}$
has finitely many resonances of $X$. Two important results here are
due to Bourgain and Dyatlov: The first, proved in \cite{bourgain_fourier_2017},
says that there is $\epsilon>0$, \emph{depending only on $\gexp$,
}such that $\gexp-\epsilon$ is an essential spectral gap of $X$.
This result is relevant if $\gexp\leq1/2$. On the other hand, it
is proved in \cite{bourgain-dyatlov_SGwithoutPressure_2018} that
$X$ has an essential spectral gap $\frac{1}{2}-\eta$ for some $\eta=\eta(\Gamma)>0$.
This result is relevant if $\gexp>1/2$. One can reformulate the conjecture
of Jakobson-Naud \cite{jakobson-naud_CriticalLine_2012} mentioned
earlier by saying that the optimal essential spectral gap of $X$
is $\frac{\gexp}{2}+\eps$ for any $\varepsilon>0$. We do not know
yet if our probabilistic techniques can be used to address high frequency
problems (i.e. resonances with large imaginary parts) and have not
attempted to do so in the present paper. However, our main result
says that this conjecture holds in the bounded frequency, large cover
regime For a broader perspective on resonances of hyperbolic surfaces
than we are able to offer here, we recommend Zworski's survey article
\cite{zworski_Survey_2017}.

\subsection{Overview of the main proof and structure of the article\label{subsec:Strateg_of_proof}}

The proof of our main result involves tools from three areas: thermodynamic
formalism, random matrices and operator algebras. Let us highlight
the main steps.

We start by recalling how one can study the resonances of deterministic
coverings of a Schottky surface with tools from the thermodynamic
formalism. Consider first a single Schottky surface $X=\Gamma\backslash\Hyp$.
Building on the symbolic coding of the---recurrent part of---geodesic
flow of $X$, the thermodynamic formalism attaches to $\Gamma$ a
family $(\Lc_{\Gamma,s})_{s\in\CC}$ of\emph{ transfer operators}.
These act on $\CC$-valued functions of a union $\Dne$\footnote{See Section \ref{subsec:Fuchsian-Schottky-groups} for the definition
of $\Dne$.} of disks in $\CC$. By choosing an appropriate functional space,
the $(\Lc_{\Gamma,s})_{s\in\CC}$ are trace-class operators on an
infinite dimensional Hilbert space $\Hc_{\Gamma}$, in particular
compact. For us, the key link between resonances and transfer operators
is the following: The resonances of $X$ on the halfplane $\{\re s>0\}$
coincide with the parameters $s$ for which 1 is an eigenvalue of
$\Lc_{\Gamma,s}$. Moreover, the multiplicities are the same.

Consider now the cover $X_{\phi}$ of $X$ associated to a homomorphism
$\phi:\Gamma\to S_{n}$, and let $\rho_{0}$ be the corresponding
unitary representation of $\Gamma$ on $V_{n}^{0}$. Any resonance
$s$ of $X$, say of multiplicity $m_{X}(s)$, is also a resonance
of $X_{\phi}$ and $m_{X_{\phi}}(s)\geq m_{X}(s)$. We say that $s$
is a \emph{new resonance} of $X_{\phi}$ if $m_{X_{\phi}}(s)>m_{X}(s)$\footnote{This definition includes the case where $m_{X}(s)=0$ or, in other
words, $s$ is not a resonance of $X$}. Twisting the classical transfer operators by $\rho_{0}$, we obtain
a family $(\Lc_{s,\rho_{0}})_{s\in\CC}$ of trace-class operators,
on a Hilbert space $\Hc_{\Gamma,n}$ of maps $\Dne\to V_{n}^{0}$,
that detect new resonances. Namely, on the halfplane $\{\re s>0\}$,
$s_{0}$ is a new resonance of $X_{\phi}$ if and only if $1$ is
an eigenvalue of $\Lc_{s_{0},\rho_{0}}$. There is a natural identification
of the space $\Hc_{\Gamma,n}$ with $\Hc_{\Gamma}\otimes V_{n}^{0}$.
Under it, the powers of the transfer operators take the form

\begin{equation}
\Lc_{s,\rho_{0}}^{\ell}=\sum_{\gamma\in\Gamma_{\ell}}T_{\gamma,s}\otimes\rho_{0}(\gamma),\label{eq:transf_op_as_polynomial_intro}
\end{equation}
where the sum is over the set $\Gamma_{\ell}$ of elements of $\Gamma$
of word length $\ell$ and the $T_{\gamma,s}$ are certain trace-class
operators of $\Hc_{\Gamma}$. To understand why random matrices enter
naturally into the proof, it is useful to think of (\ref{eq:transf_op_as_polynomial_intro})
as saying that $\Lc_{s,\rho_{0}}^{\ell}$ is a noncommutative polynomial
with operator coefficients in the variables $x_{\gamma}\eqdf\rho_{0}(\gamma),\gamma\in\Gamma_{1}$.

Let us now add the randomness to the picture: Consider the random,
degree $n$ cover $X_{n}$ of $X$ associated to a uniformly random
homomorphism $\phi_{n}:\Gamma\to S_{n}$. We now have a random unitary
representation $(\rho_{n}^{0},V_{n}^{0})$ of $\Gamma$ and random
transfer operators $\Lc_{s,\rho_{n}^{0}}$. To obtain the ``polynomial
expression'' of $\Lc_{s,\rho_{n}^{0}}^{\ell}$ one replaces $\rho_{0}$
by $\rho_{n}^{0}$ in (\ref{eq:transf_op_as_polynomial_intro}); the
coefficients $T_{\gamma,s}$ remain the same. Our strategy to show
that, most likely, a single $s_{0}$ is not a new resonance of $X_{n}$
is finding an integer $\ell>0$ such that $\opnorm{\Lc_{s,\rho_{n}^{0}}^{\ell}}<1$
with high probability. We achieve this is by feeding into Theorem
\ref{thm:sc_deterministic} the main theorem of \cite{bordenave-collins_EvRandomLifts_2019},
a deep result on random matrices---which is also a vital tool in
\cite{hide-magee_NearOptimal}---saying that the random representations
$\rho_{n}^{0}$ almost surely tend, in a suitable sense, to the left
regular representation $\rho_{\Gamma}$ of $\Gamma$. This implies,
from the ``polynomial expressions'' of $\Lc_{s,\rho_{n}^{0}}^{\ell}$,
that $\opnorm{\Lc_{s_{0},\rho_{n}^{0}}^{\ell}}$ tends in probability,
as $n\to\infty$, to the norm of the deterministic operator $\Lc_{s_{0},\rho_{\Gamma}}^{\ell}$
on $\Hc_{\Gamma}\otimes\ell^{2}(\Gamma)$, defined as in (\ref{eq:transf_op_as_polynomial_intro}),
and which we denote by $\Tc_{\ell,s_{0}}$ in the article. 

Applying an upper bound of Buchholz for the norm of an operator of
the form
\[
\sum_{\gamma\in\Gamma_{\ell}}T_{\gamma}\otimes\rho_{\Gamma}(\gamma),
\]
we show that whenever $\re s_{0}>\frac{\gexp}{2}$, $\opnorm{\Tc_{\ell,s_{0}}}\to0$
as $\ell\to\infty$ with a rate of convergence. Hence $\opnorm{\Lc_{s_{0},\rho_{n}^{0}}^{\ell}}<1$
with high probability for any big enough $\ell$. To improve the result
from a single $s_{0}$ to $s_{0}$ in a fixed compact subset of the
halfplane $\{\re s>\gexp/2\}$, we apply a rough estimate of the dependence
of the transfer operators on the parameter $s$. 

The rest of the article is divided into five sections. Section \ref{sec:Preliminaries}
covers the background on hyperbolic Schottky surfaces needed for our
main result: a geometric definition of Fuchsian Schottky group, the
model we use of random Schottky surfaces and the link between transfer
operators and resonances. Section \ref{sec:Deterministic-a-priori}
gathers several technical estimates that we use later in the article.
The goal of Section \ref{sec:Strong_convergence} is to prove Theorem
\ref{thm:sc_deterministic}: we define strong convergence, we justify
why the norm of the random transfer operators of $X_{n}$ are controlled
by that of the limit operators $\Tc_{\ell,s}$ and we state our main
norm bound for the limit operators. This last bound is Proposition
\ref{prop:norm_Tks}, and is proved in Section \ref{sec:Norm-estimation-of}.
Finally, in Section \ref{sec:Proof-sp_gap_Schottky} we state the
result of Bordenave-Collins and we complete the proof of Theorem \ref{thm:sp_gap_random_Schottky}.

\section{Preliminaries\label{sec:Preliminaries}}

This section contains the background on Schottky groups and surfaces
needed for this work: the basic definitions and notation, the description
of the model of random Schottky surfaces we consider, and the link
between resonances and transfer operators.

\subsection{Fuchsian Schottky groups\label{subsec:Fuchsian-Schottky-groups}}

Here we give a geometric definition of Fuchsian Schottky groups, we
recall their basic properties and fix the notation we use to work
with them. 

We consider the action of $\PSL(2,\RR)$ by M\"obius transformations
on the extended complex plane $\widehat{\CC}=\CC\cup\{\infty\}$ given
by 
\[
\begin{pmatrix}a & b\\
c & d
\end{pmatrix}z=\frac{az+b}{cz+d}.
\]
This action preserves the upper halfplane 
\[
\Hyp=\{z\in\CC\mid\im z>0\},
\]
which we endow with the Riemannian metric $\frac{dx^{2}+dy^{2}}{y^{2}}$
in coordinates $z=x+iy$---this is the upper halfplane model of the
hyperbolic plane---. Moreover, $\PSL(2,\RR)\curvearrowright\Hyp$
is isometric, and it identifies the group of orientation preserving
isometries of $\Hyp$ with $\PSL(2,\RR)$. If $\Gamma$ is a discrete
subgroup of $\PSL(2,\RR)$ without torsion, the quotient $\Gamma\backslash\Hyp$
is a complete, connected and orientable hyperbolic surface\footnote{By hyperbolic surface we mean a Riemannian surface of constant curvature
$-1$.}. Conversely, any such surface can be represented as $\Gamma\backslash\Hyp$,
for some $\Gamma$ as before.

The focus of this work are the hyperbolic surfaces associated to Fuchsian
Schottky groups, which are defined as follows: Consider an integer
$N\geq2$ and let $\mathcal{\Ac}=\{1,\ldots,2N\}$. If $a\in\Ac$,
we denote $\mirror a$ the unique element of $\Ac$ such that $\mirror a\equiv a+N\pmod{2N}$.
Consider open disks $(D_{a})_{a\in\Ac}$ in $\CC$ centered in $\RR$
and elements $(\gamma_{a})_{a\in\Ac}$ of $\PSL(2,\RR)$ such that
$\gamma_{\mirror a}=\gamma_{a}^{-1}$, and verifying the following
conditions: 
\begin{enumerate}
\item The closures of the $(D_{a})_{a\in\Ac}$ are pairwise disjoint;
\item The image of $\widehat{\CC}\backslash D_{\widetilde{a}}$ under $\gamma_{a}$
is the closure of $D_{a}$.
\end{enumerate}
We write $\Dne$ for the union $\cup_{a\in\Ac}D_{a}$, and $I_{a}$
for the open interval $D_{a}\cap\RR$. The group $\Gamma$ generated
by $(\gamma_{a})_{a\in\Ac}$ is free with basis $\gamma_{1},\ldots,\gamma_{N}$
by Klein's Ping-Pong Lemma, since the $(D_{a})_{a\in\Ac}$ are disjoint.
It is also a discrete subgroup of $\PSL(2,\RR)$ since $\widehat{\CC}-\Dne$
is a fundamental domain of $\Gamma\curvearrowright\widehat{\CC}$,
and condition (1) above guarantees that $\Gamma$ is convex cocompact. 

For us, a \emph{Fuchsian Schottky group }is a subgroup of $\PSL(2,\RR)$
admitting a description as the $\Gamma$ above. As proved by Button
in \cite{button_all_1998}, any convex cocompact, non cocompact Fuchsian
group is a Fuchsian Schottky group.

Let $\Gamma$ be a Fuchsian Schottky group. We denote by $N_{\Gamma}$
the rank of $\Gamma$. We assume always that $\Gamma$ comes equipped
with a fixed choice of Schottky data $(D_{a})_{a\in\Ac}$ and $(\gamma_{a})_{a\in\Ac}$.
Let $\Bc=\{\gamma_{1},\ldots,\gamma_{N_{\Gamma}}\}$. From now on
we identify each $a\in\Ac$ with the corresponding $\gamma_{a}$.
So, $\Bc$ is a basis of $\Gamma$ and $\Ac=\Bc\sqcup\Bc\inv$. Let
$\wlength{\cdot}$ be the word length on $\Gamma$ with respect to
$\Ac$. We respectively denote $\Gamma_{n}$ and $\Gamma_{\geq n}$
the sets of $\gamma\in\Gamma$ of word length $=n$ and $\geq n$.
The notation $\gamma_{1}\to\gamma_{2}$ means that $\wlength{\gamma_{1}\gamma_{2}}=\wlength{\gamma_{1}}+\wlength{\gamma_{2}}$. 

Consider $\gamma\in\Gamma_{n}$ with $n\geq1$ and its expression
$a_{1}\cdots a_{n}$ with the $a_{j}\in\Ac$. Sometimes we refer to
the $a_{j}$ as the \emph{letters of $\gamma$,} and to $\gamma$
as a \emph{word} in the alphabet $\Ac$. The first and last letters
$a_{1}$ and $a_{n}$ of $\gamma$ will be respectively denoted $\Start{\gamma}$
and $\End{\gamma}$, while $\backspace{\gamma}$ stands for the word
$a_{1}\cdots a_{n-1}$ obtained by erasing the last letter of $\gamma$.
We associate to each $\gamma$ an open disk and an interval:
\[
D_{\gamma}=\backspace{\gamma}D_{\End{\gamma}},\quad I_{\gamma}=D_{\gamma}\cap\RR.
\]
We denote by $\len{\gamma}$ the length of $I_{\gamma}$.

\subsection{Random covering spaces\label{subsec:Random-covering-spaces}}

Here we recall the model of random Schottky surface we work with,
which was introduced in \cite{magee-naud_explicit_2020}.

Consider a Fuchsian Schottky group $\Gamma$, the corresponding hyperbolic
surface $X=\Gamma\backslash\Hyp$ and let $S_{n}$ be the group of
permutations of $[n]\eqdf\{1,2,\ldots,n\}$. For any homomorphism
$\phi:\Gamma\to S_{n}$, we consider the diagonal action $\gamma(z,j)=(\gamma z,\phi(\gamma)j)$
of $\Gamma$ on $\Hyp\times[n]$. The quotient
\[
X_{\phi}\eqdf\Gamma\backslash(\Hyp\times[n])
\]
is a hyperbolic surface, and the projection $\Hyp\times[n]\to\Hyp$
induces an $n$-sheeted covering map $X_{\phi}\to X$. We denote by
$V_{n}$ the Hilbert space $\ell^{2}([n])$\footnote{This is the Hilbert space of maps $[n]\to\CC$ with the inner product
$\scal fg=\sum_{j}f(j)\overline{g(j)}.$}, and $V_{n}^{0}$ the subspace of maps $f:[n]\to\CC$ such that $\sum_{j\in[n]}f(j)=0$.
Note that $\phi$ induces unitary representations of $\Gamma$ on
$V_{n}$ and $V_{n}^{0}$.

Since a homomorphism $\Gamma\to S_{n}$ is completely determined by
its restriction to the basis $\Bc$, $\mathrm{Hom}(\Gamma,S_{n})$
is finite and hence we can talk about the uniform probability measure
on it. Throughout the article, $\phi_{n}$ is a \emph{uniformly random
homomorphism} $\Gamma\to S_{n}$---in other words, a random variable
on $\mathrm{Hom}(\Gamma,S_{n})$ with uniform distribution---and
$X_{n}$ is the corresponding random $n$-sheeted covering space of
$X$. We write $\rho_{n}$ and $\rho_{n}^{0}$ for the random unitary
representations of $\Gamma$ on $V_{n}$ and $V_{n}^{0}$ coming from
$\phi_{n}$.

\subsection{Conventions and notation for Hilbert spaces\label{subsec:Conventions-and-notation}}

Here fix the notation and conventions on Hilbert spaces we use throughout
the article. 

All the Hilbert spaces we consider are assumed to be complex. Let
$V$ and $W$ be Hilbert spaces. We denote $\scal{\cdot}{\cdot}_{V}$
the inner product of $V$---with the convention that it is $\CC$-linear
in the first coordinate and $\CC$-antilinear in the second.---,
and $\norm{\cdot}_{V}$ is the norm coming from the inner product.Let
$B(V)$ be the group of bounded linear operators on $V$. The unitary
group $U(V)$ of $V$ is the group of operators on $V$ preserving
$\scal{\cdot}{\cdot}_{V}$. We denote by $\opnorm T$ the operator
norm of any linear map $T:V\to W$ with respect to the norms of the
inner products. When $V=W$, sometimes we write $\norm T_{V}$ instead
of $\opnorm T$ to emphasize the space on which $T$ acts. We write
$T^{*}$ for the adjoint of $T$. The tensor product $V\otimes W$---suitably
completed---is also a Hilbert space with the inner product

\[
\scal{v_{1}\otimes w_{1}}{v_{2}\otimes w_{2}}_{V\otimes W}\eqdf\scal{v_{1}}{v_{2}}_{V}\scal{w_{1}}{w_{2}}_{W}.
\]
Recall that a unitary representation of a group $\Gamma$ on $V$
is a homomorphism $\rho:\Gamma\to U(V)$. We often write the unitary
representations as a pair $(\rho,V)$.

For any $r\geq1$, we consider $\CC^{r}$ as Hilbert space with the
standard inner product 
\[
\scal{(z_{1},\ldots,z_{r})}{(w_{1},\ldots,w_{r})}=\sum_{j}z_{j}\overline{w_{j}}.
\]
We usually write $M_{r}(\CC)$ instead of $B(\CC^{r})$.

\subsection{Transfer operators and their relation to resonances\label{subsec:Transfer-operators-and}}

Here we introduce transfer operators, the main tool we use to detect
resonances of Schottky surfaces. These are a family $(\Lc_{s,\rho})_{s\in\CC}$
of operators associated to any unitary representation $(\rho,V)$
of a Fuchsian Schottky group $\Gamma$, acting on a certain space
$\Fc_{\rho}$ of functions $\Dne\to V$. The relation between transfer
operators and resonances we use to prove our main result is stated
in Proposition \ref{prop:relation_eigenvalues_transf_op}: when $\rho$
comes from a homomorphism $\phi:\Gamma\to S_{n}$, we will use the
$\Lc_{s,\rho}$ to study the resonances of $X_{\phi}$. In order to
do so, one has to choose a functional space $\Fc_{\rho}$ where the
transfer operators behave well. Like in \cite{magee-naud_explicit_2020},
our choice of $\Fc_{\rho}$ is a Bergman space. We give also the ``polynomial
expression'' of the powers of $\Lc_{s,\rho}$ with operator coefficients,
which is key for the proof of our main result. 

Let $\Gamma$ be a Fuchsian Schottky group and let $V$ be a Hilbert
space of finite dimension. For any unitary representation $(\rho,V)$
of $\Gamma$ and all $s\in\CC$, we define the operator $\Lc_{s,\rho}$
on measurable functions $f:\Dne\to V$ by 
\begin{equation}
\Lc_{s,\rho}f(z_{a})\eqdf\sum_{\substack{w\in\Gamma_{2}\\
\End w=a
}
}\backspace w'(z_{a})^{s}\rho(\backspace w\inv)f(\backspace wz_{a}).\label{eq:def_transf_op}
\end{equation}
We clarify two points about this definition: the subindex in $z_{a}$
means that $z_{a}$ belongs to $D_{a}$, and $\backspace w'(z_{a})^{s}\eqdf\exp(s\tau(\backspace w'(z_{a})))$,
where $\tau$ is the branch of the logarithm
\begin{equation}
\CC-(-\infty,0]\to\RR\oplus i(-\pi,\pi).\label{eq:branch_log}
\end{equation}

The Bergman space $\Hc(\Dne,V)$ is the Hilbert space of square integrable\footnote{With respect to the Lebesgue measure of $\Dne$.},
holomorphic maps $\Dne\to V$ with the inner product
\[
\scal{f_{1}}{f_{2}}\eqdf\int_{\Dne}\scal{f_{1}(z)}{f_{2}(z)}_{V}\dd z.
\]
Note that $\Hc(\Dne,V)$ is stable under $\Lc_{s,\rho}$. In this
work, we take $\Hc(\Dne,V)$ as the domain of the transfer operators.
This guarantees that $\Lc_{s,\rho}$ is trace-class for any $s\in\CC$---see
\cite[Corollary 4.2]{magee-naud_explicit_2020}---, so compact in
particular.

Here is the relation between resonances and transfer operators that
we use to establish our main result. It follows from the correspondence
between resonances of a hyperbolic Schottky surface and the zeros
of its Selberg zeta function---see \cite[Theorem 10.1]{borthwick_spectral_2016}---,
and the factorization \cite[Proposition 4.4 (2)]{magee-naud_explicit_2020}
of the Selberg zeta function of $X_{\phi}$.
\begin{prop}
\label{prop:relation_eigenvalues_transf_op}Let $\phi$ be a homomorphism
from a Fuchsian Schottky group $\Gamma$ to $S_{n}$, and let $\rho_{0}$
be the corresponding unitary representation of $\Gamma$ on $V_{n}^{0}$.
If $s\in\CC$ is a resonance of $X_{\phi}$ with greater multiplicity
than in $X$, then 1 is an eigenvalue of $\Lc_{s,\rho_{0}}$.
\end{prop}

Now we give a convenient formula for the powers of $\Lc_{s,\rho}$,
and then the ``polynomial expression'' of $\Lc_{s,\rho}^{\ell}$
with operator coefficients. 

We claim that for any integer $\ell>0$ and all $f\in\Hc(\Dne,V)$
we have
\begin{equation}
\Lc_{s,\rho}^{\ell}f(z_{a})=\sum_{\substack{w\in\Gamma_{\ell+1}\\
\End w=a
}
}\backspace w'(z_{a})^{s}\rho(\backspace w\inv)f(\backspace wz_{a}).\label{eq:powers_of_transf_op}
\end{equation}
Indeed, consider $w\in\Gamma_{\ell+1}$, say $w=a_{1}\cdots a_{\ell+1}$
with $a_{j}\in\Ac$. We denote $a_{j}a_{j+1}\cdots a_{\ell}$ by $w_{j\rightarrow}$.
For any $z\in D_{a_{\ell+1}}$, let
\begin{equation}
\theta(w,z)\eqdf\sum_{j=1}^{\ell}\tau(a_{j}'(w_{j+1\rightarrow}z)).\label{eq:def_theta}
\end{equation}
Iterating (\ref{eq:def_transf_op}) one obtains
\begin{equation}
\Lc_{s,\rho}^{\ell}f(z_{a})=\sum_{\substack{w\in\Gamma_{\ell+1}\\
\End w=a
}
}\exp(s\theta(w,z_{a}))\rho(\backspace w\inv)f(\backspace wz_{a}),\label{eq:powers_of_transf_op_bis}
\end{equation}
for any $a\in\Ac$. When $z_{a}$ is real, the terms $a_{j}'(w_{j+1\rightarrow}z_{a})$
in (\ref{eq:def_theta}) are real as well, so
\[
\theta(w,z_{a})=\tau\left(\prod_{j=1}^{\ell-1}a_{j}'(w_{j+1\rightarrow}z_{a})\right)=\tau(\backspace w'z_{a}).
\]
Since $\theta(w,\cdot)$ and $\tau\circ\backspace w'$ coincide in
$I_{a}$ and are holomorphic on $D_{a}$, in fact they coincide on
$D_{a}$. Hence $\exp(s\theta(w,z_{a}))=\backspace w'(z_{a})^{s}$
for all $s\in\CC$, and (\ref{eq:powers_of_transf_op}) follows.

The ``polynomial expression'' of $\Lc_{s,\rho}$ and its powers
is essentially due to the following isomorphism of Hilbert spaces:
For any $\varphi\in\Hc(\Dne,V)$ and any $v\in V$, the map $\varphi_{v}(z)\eqdf\scal{\varphi(z)}v$
lies in $\Hc(\Dne)$. Consider an orthonormal basis $(v_{i})_{i}$
of $V$. The map
\begin{equation}
\varphi\mapsto\sum_{i}\varphi_{v_{i}}\otimes v_{i}\label{eq:isomorphism_Bergmann_spaces}
\end{equation}
is an isomorphism $\Hc(\Dne,V)\to\Hc(\Dne)\otimes V$, and it does
not depend on the choice of orthonormal basis of $V$. Here are the
coefficients of the ``polynomial expression'' of $\Lc_{s,\rho}^{\ell}$:
For any $w\in\Gamma_{\geq2}$ and any $s\in\CC$ we define $\widetilde{M}_{w,s}:\Hc(D_{S(w)})\to\Hc(D_{E(w)})$
by
\begin{equation}
\widetilde{M}_{w,s}\psi_{S(w)}(z_{E(w)})\eqdf\backspace w'(z_{\End w})^{s}\psi_{S(w)}(\backspace wz_{E(w)}),\label{eq:def_Mws}
\end{equation}
with the $s$-th power defined as in (\ref{eq:def_transf_op}). Let
$\Ptt_{a}$ be the orthogonal projection $\Hc(\Dne)\to\Hc(D_{a})$.
The operators 
\[
M_{w,s}\eqdf\Ptt_{\End w}^{*}\widetilde{M}_{w,s}\Ptt_{\Start w}
\]
are trace-class, as is shown in \cite[Lemma 15.7]{borthwick_spectral_2016}.
\begin{lem}
\label{lem:transf_op_as_sum_of_tensors}Let $(\rho,V)$ be a unitary
representation of $\Gamma$. For any integer $\ell\geq1$, under the
isomorphism $\Hc(\mathbf{D},V)\simeq\Hc(\Dne)\otimes V$ given by
(\ref{eq:isomorphism_Bergmann_spaces}), we have
\begin{equation}
\Lc_{s,\rho}^{\ell}=\sum_{w\in\Gamma_{\ell+1}}M_{w,s}\otimes\rho(\backspace w\inv).\label{eq:transf_op_as_sum_of_tensors}
\end{equation}
\end{lem}

\begin{proof}
We identify $\Hc(\Dne,V)$ and $\Hc(\Dne)\otimes V$ via (\ref{eq:isomorphism_Bergmann_spaces}).
Let us denote by $\Kc_{s,\rho}^{(\ell)}$ the operator on the right-hand-side
of (\ref{eq:transf_op_as_sum_of_tensors}). It suffices to show that
$\Lc_{s,\rho}^{\ell}$ and $\Kc_{s,\rho}^{(\ell)}$ agree on all the
$f\in\Hc(\Dne,V)$ of the form $\varphi\otimes v$, with $\varphi\in\Hc(\Dne$)
and $v$ any unit vector in $V$. For any such $f$ and all $b\in\Ac$,
by (\ref{eq:powers_of_transf_op_bis}) we have

\begin{align*}
[\Lc_{s,\rho}^{\ell}f](z_{b}) & =\sum_{\substack{w\in\Gamma_{\ell+1}\\
\End w=b
}
}\backspace w'(z_{b})^{s}\rho(\backspace w\inv)f(\backspace wz_{b})\\
 & =\sum_{\substack{w\in\Gamma_{\ell+1}\\
\End w=b
}
}\backspace w'(z_{b})^{s}\varphi(\backspace wz_{b})\rho(\backspace w\inv)v\\
 & =\left(\sum_{\substack{w\in\Gamma_{\ell+1}\\
\End w=b
}
}M_{w,s}\otimes\rho(\backspace w\inv)\right)f(z_{b})\\
 & =[\Ptt_{b}\Kc_{s,\rho}^{(\ell)}f](z_{b}),
\end{align*}
which proves our claim.
\end{proof}

\section{Deterministic a priori bounds\label{sec:Deterministic-a-priori}}

In this section we gather technical lemmas about Schottky groups,
transfer operators and Bergman spaces. Some of them are well known,
while others are specifically tailored for our needs. We suggest the
reader to skim the statements in a first reading, and come back to
the details when the later sections demand it. 

We start with four properties of $w\mapsto\Upsilon_{w}$, whose proofs
can be respectively found in \cite[Lemma 3.4, Lemma 3.5, Lemma 3.2, Lemma 3.1]{magee-naud_explicit_2020}.
\begin{lem}
\label{lem:coarse-homomorphism}Let $\Gamma$ be a Fuchsian Schottky
group. The following holds:
\begin{description}
\item [{(Rough~multiplicativity)}] For all $w_{1},w_{2}\in\Gamma_{\geq1}$
with $w_{1}\rightarrow w_{2}$ we have
\begin{equation}
\len{w_{1}w_{2}}\asymG\len{w_{1}}\len{w_{2}}.\label{eq:upsilon-concatenation}
\end{equation}
\item [{(Mirror~estimate)}] For all $w\in\Gamma_{\geq1}$ we have
\begin{equation}
\len w\asymG\len{w\inv}.\label{eq:upsilon-mirror-eq}
\end{equation}
\item [{(Derivatives)}] For all $w\in\Gamma_{\geq2}$ and any $z\in D_{\End w}$
we have
\begin{equation}
|\backspace w'(z)|\asymG\len w.\label{eq:upsilon-deriv-eq}
\end{equation}
\item [{(Exponential~Bound)}] There are constants $\theta_{\Gamma},\tau_{\Gamma}\in(0,1)$
such that for all $\gamma\in\Gamma_{\geq1}$ we have 
\begin{equation}
\theta_{\Gamma}^{\wlength{\gamma}}\llG\len{\gamma}\llG\tau_{\Gamma}^{\wlength{\gamma}}.\label{eq:upsilon-exp-bound}
\end{equation}
\end{description}
\end{lem}

In the next lemma we consider $\tau_{\Gamma}$ as in (\ref{eq:upsilon-exp-bound}).
\begin{lem}
\label{lem:exp_sums_of_boundary_lengths}Let $\Gamma$ be a Fuchsian
Schottky group. For any $\eps\in(0,1)$ and any integer $k\geq1$
we have
\[
\sum_{w\in\Gamma_{k}}\len w^{\gexp+\eps}\llG\tau_{\Gamma}^{k\eps}.
\]
\end{lem}

\begin{proof}
Consider $w\in\Gamma_{k}$ and $\eps\in(0,1)$. By (\ref{eq:upsilon-exp-bound})
we have $\len w\llG\tau_{\Gamma}^{k}$. \cite[Lemma 2.11]{bourgain_fourier_2017}
tells us that $\len w^{\gexp}\asymG\mu(I_{w})$, where $\mu$ is a
Patterson-Sullivan probability measure on $\partial\Hyp$ associated
to $\Gamma$. Then
\begin{align*}
\sum_{w\in\Gamma_{k}}\len w^{\gexp+\eps} & \llG\tau_{\Gamma}^{k\eps}\sum_{w\in\Gamma_{k}}\mu(I_{w})=\tau_{\Gamma}^{k\eps},
\end{align*}
which completes the proof. In the last step we used the fact that
$\sum_{w\in\Gamma_{k}}\mu(I_{w})=1$ \footnote{The probability measure $\mu$ is supported on the limit set of $\Gamma$,
which is covered by the pairwise disjoint intervals $(I_{w})_{w\in\Gamma_{k}}$.}. 
\end{proof}
We move on to two auxiliary results about Bergman spaces that we use
later. Both follow easily form the fact that the Bergman space of
an open domain $\Omega$ in $\CC$ has a \emph{reproducing kernel}
$B_{\Omega}$. Namely, $B_{\Omega}$ is a map $\Omega\times\Omega\to\CC$
respectively holomorphic and antiholomorphic in the first and second
coordinates, such that for any $f\in\Hc(\Omega)$ and all $z_{0}\in\Omega$
we have
\[
f(z_{0})=\int_{\Omega}B_{\Omega}(z_{0},z)f(z)\dd z.
\]
For an open disk $D\subseteq\CC$, say of radius $r$ and center $c$,
there is the explicit formula

\begin{equation}
B_{D}(z_{1},z_{2})=\frac{r^{2}}{\pi[r^{2}-(z_{1}-c)(\overline{z_{2}-c})]^{2}}.\label{eq:Bergman_kernel}
\end{equation}
See \cite[p. 378]{borthwick_spectral_2016}.
\begin{lem}
\label{lem:estimate_of_restriction_Bergmann_spaces}Let $D$ and $D'$
be open disks in $\CC$, respectively with centers $c,c'\in\RR$ and
radii $r,r'$. Suppose that $D$ contains the closure of $D'$. For
any holomorphic function $\psi:D\to\CC$ we have
\[
\norm{\psi}_{\Hc(D')}\leq\frac{r'}{\pi^{\frac{1}{2}}(r-r'-|c-c'|)^{2}}\norm{\psi}_{\Hc(D)}.
\]
\end{lem}

\begin{proof}
Without loss of generality we assume that $c=0$. We compute the norm
of $\psi$ in $\Hc(D')$ with the aid of $B_{D}$: For any $z_{0}\in D$
we have
\[
\psi(z_{0})=\int_{D}B_{D}(z_{0},z)\psi(z)\dd z.
\]
Any $z'\in D'$ verifies $|\overline{z'}|=|z'|\leq r'+|c'|$, so from
(\ref{eq:Bergman_kernel}) we readily see that
\[
|B_{D}(z',z)|\leq\frac{1}{\pi(r-r'-|c'|)^{2}}
\]
for any $z\in D$. Hence,

\begin{align*}
\int_{D'}|\psi(z')|^{2}\dd z' & =\int_{D'}\left|\int_{D}B_{D}(z',z)\psi(z)\dd z\right|^{2}\dd z'\\
 & \leq\frac{1}{\pi^{2}(r-r'-|c'|)^{4}}\int_{D'}\left(\int_{D}|\psi(z)|\dd z\right)^{2}\dd z'\\
 & \leq\frac{1}{\pi^{2}(r-r'-|c'|)^{4}}\int_{D'}\norm{\psi}_{\Hc(D)}^{2}\dd z'\\
 & =\frac{(r')^{2}}{\pi(r-r'-|c'|)^{4}}\norm{\psi}_{\Hc(D)}^{2},
\end{align*}
and the result follows. From line 2 to line 3 in the preceding computation
we used Jensen's inequality.
\end{proof}
For ease of reference we record the following consequence of Lemma
\ref{lem:estimate_of_restriction_Bergmann_spaces}.
\begin{cor}
\label{cor:norm_restriction_Bergman}Let $\Gamma$ be a Fuchsian Schottky
group. For any $w\in\Gamma_{\geq2}$ and any $\psi\in\Hc(D_{S(w)})$
we have
\[
\norm{\psi}_{\Hc(D_{w})}\llG\len w\norm{\psi}_{\Hc(D_{S(w)}).}
\]
\end{cor}

\begin{proof}
For any $\gamma\in\Gamma-\{I\}$, we denote respectively the center
and the radius of $D_{\gamma}$ by $c_{\gamma}$ and $r_{\gamma}$.
Let $a=S(w)$. Since $D_{w}$ is contained in $D_{a}$, by Lemma \ref{lem:estimate_of_restriction_Bergmann_spaces}
we have

\begin{equation}
\norm{\psi}_{\Hc(D_{w})}\leq\frac{r_{w}}{\pi^{\frac{1}{2}}(r_{a}-r_{w}-|c_{a}-c_{w}|)^{2}}\norm{\psi}_{\Hc(D_{a})}.\label{eq:Z1}
\end{equation}
Note that $r_{w}+|c_{a}-c_{w}|$ is the radius of the smallest circle
with center $c_{a}$ containing $D_{w}$, hence $r_{a}-r_{w}-|c_{a}-c_{w}|\ggG1$.
Also, $2r_{w}=\len w$. The claim follows from these observations
and (\ref{eq:Z1}).
\end{proof}
\begin{lem}
\label{lem:norm_evaluation_map}Let $\Gamma$ be a Fuchsian Schottky
group and let $V$ be a Hilbert space of finite dimension. For any
$f\in\Hc(\Dne,V)$, any $w\in\Gamma_{\geq2}$ and all $z_{w}\in D_{w}$
we have
\[
\norm{f(z_{w})}_{V}\llG\norm f_{\Hc(\Dne,V)}.
\]
\end{lem}

\begin{proof}
We start with the case $V=\CC$. Consider $\varphi\in\Hc(\Dne)$ and
let $a=\Start w$. Let $c_{a},r_{a}$ and $B_{a}$ be respectively
the center, radius and the Bergman kernel of $D_{a}$. Since $z_{w}$
lies in $D_{a}$, we have $\varphi(z_{w})=\scal{\Ptt_{a}\varphi}{\overline{B_{a}(z_{w},\cdot)}}_{\Hc(D_{a})}$. 

Let
\[
\Bc_{a}\eqdf\bigcup_{\substack{\gamma\in\Gamma_{2}\\
\Start{\gamma}=a
}
}D_{\gamma}.
\]
Note that there is $\eps_{\Gamma}>0$ such that for any $z\in\Bc_{a}$
we have $|z-c_{a}|\leq r_{a}-\eps_{\Gamma}$\footnote{Since the closure of $\Bc_{a}$ is contained in $D_{a}$.}.
It follows then from (\ref{eq:Bergman_kernel}) that $|B_{a}(z_{w},z_{a})|\llG1$
for any $z_{a}\in D_{a}$. The Cauchy-Schwarz inequality on $\Hc(D_{a})$
gives

\begin{align}
|\varphi(z_{w})| & \leq\norm{\overline{B_{a}(z_{w},\cdot)}}_{\Hc(D_{a})}\norm{\Ptt_{a}\varphi}_{\Hc(D_{a})}\nonumber \\
 & \llG\norm{\varphi}_{\Hc(\Dne)},\label{eq:ev0}
\end{align}
which completes the proof in this case.

For a general $V$ and $f\in\Hc(\Dne,V)$, we consider an orthonormal
basis $v_{1},\ldots,v_{d}$ of $V$ and we express $f$ as $\sum_{j}\varphi_{j}\otimes v_{j}$
for some $\varphi_{j}\in\Hc(\Dne)$. To complete the proof we apply
(\ref{eq:ev0}) as follows:
\begin{align*}
\norm{f(z_{w})}_{V} & =\left(\sum_{j=1}^{d}|\varphi_{j}(z_{w})|^{2}\right)^{\frac{1}{2}}\\
 & \llG\left(\sum_{j=1}^{d}\norm{\varphi_{j}}_{\Hc(\Dne)}^{2}\right)^{\frac{1}{2}}=\norm f_{\Hc(\Dne,V)}.
\end{align*}
\end{proof}
We close this section with an estimate of the norm variation of a
transfer operator with respect to the parameter $s$. Recall that
$N_{\Gamma}$ is the rank of a Schottky group $\Gamma$. 
\begin{lem}
\label{lem:transf_op_dependence_on_s}Let $\Gamma$ be a Fuchsian
Schottky group and let $\Kc$ be a compact subset of $\CC$. There
are positive constants $B_{\text{\ensuremath{\Gamma}}},J_{\Gamma}$
and $C_{\Kc}$ such that for any $\ell\ggG1$, any finite dimensional
unitary representation $(\rho,V)$ of $\Gamma$ and any $s_{1},s_{2}\in\Kc$
we have
\[
\opnorm{\Lc_{s_{1},\rho}^{\ell}-\Lc_{s_{2},\rho}^{\ell}}\leq J_{\Gamma}|s_{1}-s_{2}|(\ell+1)\left[(2N_{\Gamma}-1)C_{\Kc}^{B_{\Gamma}}\right]^{\ell+1}.
\]
\end{lem}

\begin{proof}
Consider $f\in\Hc(\Dne,V)$, $b\in\Ac$ and $z_{b}\in D_{b}$. From
(\ref{eq:powers_of_transf_op}) we see that 
\begin{equation}
(\Lc_{s_{1},\rho}^{\ell}-\Lc_{s_{2},\rho}^{\ell})f(z_{b})=\sum_{\substack{w\in\Gamma_{\ell+1}\\
\End w=b
}
}\left(\backspace w'(z_{b})^{s_{1}}-\backspace w'(z_{b})^{s_{2}}\right)\rho(\backspace w\inv)f(\backspace wz_{b}).\label{eq:W0}
\end{equation}
In all the proof, $w$ is an element of $\Gamma_{\ell+1}$ with $\End w=b$.
By (\ref{eq:upsilon-deriv-eq}) and (\ref{eq:upsilon-exp-bound}),
both $|\backspace w'(z_{b})|$ and $\len w$ are $<1$ whenever $\ell\ggG1$.
We assume this for the rest of the proof.

To estimate the norm of $(\Lc_{s_{1},\rho}^{\ell}-\Lc_{s_{2},\rho}^{\ell})f(z_{b})$
we will use the next elementary bound, valid for all $z_{1},z_{2}\in\CC$:
\[
\left|e^{z_{1}}-e^{z_{2}}\right|\leq\vert z_{1}-z_{2}\vert e^{\max\left\{ \re z_{1},\re z_{2}\right\} }.
\]
Thus
\begin{equation}
|\backspace w'(z_{b})^{s_{1}}-\backspace w'(z_{b})^{s_{2}}|\leq|s_{1}-s_{2}||\tau(\backspace w'z_{b})|e^{\max\{\re[s_{1}\tau(\backspace w'z_{b})],\re[s_{2}\tau(\backspace w'z_{b})]\}},\label{eq:ABC0}
\end{equation}
where $\tau$ is the branch of the logarithm as in (\ref{eq:branch_log}).
Since $|\backspace w'z_{b}|<1$, then $|\tau(\backspace w'z_{b})|\leq\pi+\log(|\backspace w'z_{b}|\inv)$,which
by (\ref{eq:upsilon-deriv-eq}) and (\ref{eq:upsilon-exp-bound})
implies that
\begin{equation}
|\tau(\backspace w'z_{b})|\leq B_{\Gamma}(\ell+1)\label{eq:ABC1}
\end{equation}
for some $B_{\Gamma}>0$. Let $C_{\Kc}\eqdf\max_{z\in\Kc}e^{|z|}.$
By (\ref{eq:ABC1}), we have the following for any $s\in\Kc$:
\begin{align}
\exp[\re(s\tau(\backspace w'z_{b}))] & \leq\exp|s\tau(\backspace w'z_{b})|\nonumber \\
 & \leq C_{\Kc}^{B_{\Gamma}(\ell+1)}.\label{eq:ABC2}
\end{align}
Plugging (\ref{eq:ABC1}) and (\ref{eq:ABC2}) into (\ref{eq:ABC0})
yields
\begin{equation}
|\backspace w'(z_{b})^{s_{1}}-\backspace w'(z_{b})^{s_{2}}|\llG|s_{1}-s_{2}|(\ell+1)C_{\Kc}^{B_{\Gamma}(\ell+1)}\label{eq:ABC3}
\end{equation}

Applying the triangle inequality in (\ref{eq:W0}) followed by (\ref{eq:ABC3})
and Lemma \ref{lem:norm_evaluation_map}, we get

\begin{align}
\norm{(\Lc_{s_{1},\rho}^{\ell}-\Lc_{s_{2},\rho}^{\ell})f(z_{b})}_{V} & \llG|s_{1}-s_{2}|(\ell+1)C_{\Kc}^{B_{\Gamma}(\ell+1)}\norm f_{\Hc(\Dne,V)}\frac{\#\Gamma_{\ell+1}}{2N_{\Gamma}}\nonumber \\
 & \llG|s_{1}-s_{2}|(\ell+1)\left[C_{\Kc}^{B_{\Gamma}}(2N_{\Gamma}-1)\right]^{\ell+1}\norm f_{\Hc(\Dne,V)},\label{eq:}
\end{align}
and since
\[
\norm{(\Lc_{s_{1},\rho}^{\ell}-\Lc_{s_{2},\rho}^{\ell})f}_{\Hc(\Dne,V)}=\left(\int_{\Dne}\norm{(\Lc_{s_{1},\rho}^{\ell}-\Lc_{s_{2},\rho}^{\ell})f(z)}_{V}^{2}dz\right)^{\frac{1}{2}},
\]
we are done.
\end{proof}

\section{Strong convergence and the proof of Theorem \ref{thm:sc_deterministic}
\label{sec:Strong_convergence}}

This section has two parts: in Section \ref{subsec:Strong-convergence-and_mat_amp}
we introduce the notion of strong convergence for random unitary representations
of a free group $\Gamma$, which is in fact a condition on the associated
representations of the group algebra $\CC[\Gamma]$. We also explain
how to upgrade the strong convergence via the matrix amplification
trick in order to handle the random transfer operators of Theorem
\ref{thm:sc_deterministic}. The proof of this last result is completed
in Section \ref{subsec:proof_sc_deter} by using a norm bound that
we prove in Section \ref{sec:Norm-estimation-of}. 

\subsection{Strong convergence and matrix amplification\label{subsec:Strong-convergence-and_mat_amp}}

Suppose in this section that $\Gamma$ is a free group of rank $N$
and $(\rho_{j},W_{j})_{j\geq1}$ is a sequence of (possibly random)
unitary representations of $\Gamma$. 

Let us recall two notions of convergence for random variables: We
say that a sequence of complex random variables $(Z_{j})_{j\geq1}$
\emph{converges asymptotically almost surely} (a.a.s.) to a constant
$z\in\CC$ if for all $\epsilon>0$,
\[
\lim_{j\to\infty}\PP(|Z_{j}-z|>\epsilon)=0.
\]
This is the same as convergence in probability. We say additionally
that the convergence is almost sure (a.s.) if for all $\epsilon>0$
we have
\[
\sum_{j=1}^{\infty}\PP(|Z_{j}-z|>\epsilon)<\infty.
\]

\begin{defn}[Strong convergence]
A sequence of unitary representations $(\rho_{j},W_{j})_{j\geq1}$
of $\Gamma$ \emph{strongly converges }to $(\rho_{\infty},W_{\infty})$
if for all $f\in\CC[\Gamma]$---the group algebra of $\Gamma$---we
have
\begin{equation}
\lim_{j\to\infty}\norm{\rho_{j}(f)}_{W_{j}}=\norm{\rho_{\infty}(f)}_{W_{\infty}}.\label{eq:def_strong_convergence}
\end{equation}
When the unitary representations $(\rho_{j})_{j\geq1}$ are random,
we say that they a.a.s. (a.s.) strongly converge to $\rho_{\infty}$
if the convergence in (\ref{eq:def_strong_convergence}) is a.a.s.
(a.s.).
\end{defn}

In the rest of this section, we assume that the random unitary representations
$(\rho_{j},W_{j})_{j\geq1}$ of $\Gamma$ are finite dimensional and
that they a.a.s. or a.s. strongly converge to the left regular representation
$(\rho_{\Gamma},\ell^{2}(\Gamma))$ of $\Gamma$. 

By matrix amplification (see for example \cite[Section 9]{haagerup-thorb_ANewApplication_2005}),
the strong convergence of the sequence $(\rho_{j})_{j\geq1}$ implies
the following more general statement. 
\begin{prop}
\label{prop:Bordenave-Collins+linearization} For any $r\geq1$ and
any finitely supported map $\gamma\mapsto F_{\gamma}$ with values
in $M_{r}(\CC)$ we have
\[
\lim_{j\to\infty}\norm{\sum_{\gamma\in\Gamma}F_{\gamma}\otimes\rho_{j}(\gamma)}_{\CC^{r}\otimes W_{j}}=\norm{\sum_{\gamma\in\Gamma}F_{\gamma}\otimes\rho_{\Gamma}(\gamma)}_{\CC^{r}\otimes\ell^{2}(\Gamma)}
\]
 (a.s or a.a.s., respectively, if the representations are random).
\end{prop}

We need a mild extension of Proposition \ref{prop:Bordenave-Collins+linearization}
allowing the $F_{\gamma}$ to be compact operators an infinite dimensional
Hilbert space, which is obtained via a standard approximation argument
(see, for example, \cite[proof of Prop. 6.2]{magee-thomas_StronglyConvergent_2023}). 
\begin{cor}
\label{cor:Bordenave_Collins_plus_epsilon}For any finitely supported
map $\gamma\mapsto T_{\gamma}$ on $\Gamma$ taking values in the
compact operators of a separable Hilbert space $\Hc$, we have
\[
\lim_{i\to\infty}\norm{\sum_{\gamma\in\Gamma}T_{\gamma}\otimes\rho_{i}(\gamma)}_{W\otimes\CC^{n_{i}}}=\norm{\sum_{\gamma\in\Gamma}T_{\gamma}\otimes\rho_{\Gamma}(\gamma)}_{W\otimes\ell^{2}(\Gamma)}
\]
 (a.s or a.a.s., respectively, if the representations are random).
\end{cor}

\subsection{The proof of Theorem \ref{thm:sc_deterministic}\label{subsec:proof_sc_deter}}

In this section we consider a Fuchsian Schottky group $\Gamma$ and
a sequence $(\rho_{j})_{j\geq1}$ of random unitary representations
of $\Gamma$ a.a.s. strongly converging to $\rho_{\Gamma}$. Recall---see
(\ref{eq:powers_of_transf_op})---that the powers of the transfer
operators can be written as
\[
\Lc_{s,\rho_{j}}^{\ell}=\sum_{w\in\Gamma_{\ell+1}}M_{w,s}\otimes\rho_{j}(\backspace w\inv).
\]
By Corollary \ref{cor:Bordenave_Collins_plus_epsilon} we know that,
for big $n$, the size of $\Lc_{s,\rho_{i}}^{\ell}$ is governed a.a.s.
by that of the \emph{limit operator}
\begin{equation}
\Tc_{\ell,s}\eqdf\sum_{w\in\Gamma_{\ell+1}}M_{w,s}\otimes\rho_{\Gamma}(\backspace w\inv)\label{eq:def_Tls}
\end{equation}
on $\Hc(\Dne)\otimes\ell^{2}(\Gamma)$. We will prove Theorem \ref{thm:sc_deterministic}
by showing that we can choose $\ell=\ell(\Gamma,\Kc)$ such that $\opnorm{\Tc_{\ell,s}}<1$
for any $s\in\Kc$.

The key in the next norm bound of $\Tc_{\ell,s}$, whose proof is
deferred to Section \ref{sec:Norm-estimation-of}.
\begin{prop}
\label{prop:norm_Tks}Let $\Gamma$ be a Fuchsian Schottky group.
There are constants $A_{\Gamma}>0,C_{\Gamma}>1$ and $\tau_{\Gamma}\in(0,1)$
with the next property: for any integer $\ell\ggG1$, any $\eps>0$
and any $s\in\CC$ with $\re s\geq\frac{\gexp}{2}+\eps$ we have
\[
\opnorm{\Tc_{\ell,s}}\leq A_{\Gamma}C_{\Gamma}^{\re s}e^{\pi|\im s|}(\ell+1)\tau_{\Gamma}^{\eps\ell}.
\]
\end{prop}

Taking Proposition \ref{prop:norm_Tks} momentarily for granted, we
establish now Theorem \ref{thm:sc_deterministic}.

\begin{proof}[Proof of Theorem \ref{thm:sc_deterministic}]
Consider $A_{\Gamma}>0,C_{\Gamma}>1$ and $\tau_{\Gamma}\in(0,1)$
as in Proposition \ref{prop:norm_Tks}. It suffices to prove the result
when $\Kc$ is a rectangle of the form
\[
\Kc=\left\{ s\in\CC\mid\frac{\gexp}{2}+\eps\leq\re s\leq M_{r},|\im s|\leq M_{i}\right\} ,
\]
with $\eps\in\left(0,1\right)$ and $M_{r},M_{i}>1$. We fix these
three parameters for the rest of the proof, as well as $\ell=\ell(\Kc,\Gamma)$
big enough so that Lemma \ref{lem:transf_op_dependence_on_s} and
Proposition \ref{prop:norm_Tks} apply to it, and
\begin{equation}
A_{\Gamma}C_{\Gamma}^{M_{r}}e^{\pi M_{i}}(\ell+1)\tau_{\Gamma}^{\eps\ell}<1-2\eps.\label{eq:FIN0}
\end{equation}
Consider also 
\[
\eps_{1}=\eps_{1}(\Gamma,\Kc)\eqdf\frac{\eps}{2J_{\Gamma}(\ell+1)\left[(2N_{\Gamma}-1)C_{\Kc}^{B_{\Gamma}}\right]^{\ell+1}},
\]
where $J_{\Gamma},B_{\Gamma}$ and $C_{\Kc}$ are as in Lemma \ref{lem:transf_op_dependence_on_s}. 

Consider, as in the statement, a sequence $(\rho_{j},W_{j})_{j\ge1}$
a sequence of finite dimensional unitary representations of $\Gamma$
that a.a.s. strongly converges to $\rho_{\Gamma}$. Recall that under
the isomorphism $\Hc(\Dne,W_{j})\to\Hc(\Dne)\otimes W_{j}$ given
by (\ref{eq:isomorphism_Bergmann_spaces}) we have
\[
\Lc_{s,\rho_{j}}^{\ell}=\sum_{w\in\Gamma_{\ell}}T_{w,s}\otimes\rho_{j}(w),
\]
with $T_{w,s}$ as in (\ref{eq:def_Txs})---see Lemma \ref{lem:transf_op_as_sum_of_tensors}---. 

We denote by $\mathscr{E}_{s,j}$ the event that 
\[
\opnorm{\Lc_{s,\rho_{j}}^{\ell}}\leq\opnorm{\Tc_{\ell,s}}+\eps,
\]
where $\Tc_{\ell,s}$ is the limiting operator on $\Hc(\Dne)\otimes\ell^{2}(\Gamma)$
defined in (\ref{eq:Limit_op}). Each $T_{w,s}$ is a compact operator
of $\Hc(\Dne)$ since it is a sum of $M_{w\inv b,s}$, which are compact
as is explained right after (\ref{eq:def_Mws}).

The bound of the norm of $\Tc_{\ell,s}$ of Proposition \ref{prop:norm_Tks}
and (\ref{eq:FIN0}) imply that for any $s\in\Kc$, on $\mathscr{E}_{s,j}$
we have
\begin{equation}
\opnorm{\Lc_{s,\rho_{j}}^{\ell}}<1-\eps.\label{eq:FIN1}
\end{equation}
Let $\mathtt{S}$ be a finite subset of $\Kc$ such that the open
disks with center in $\mathtt{S}$ and radius $\eps_{1}$ cover $\Kc$,
and let $\mathscr{E}_{j}=\cap_{\mathtt{s}\in\mathtt{S}}\mathscr{E}_{\mathtt{s},j}$. 

Since $\mathtt{S}$ is finite, then $\PP(\mathscr{E}_{j})\to1$ as
$j\to\infty$. Consider any $s\in\Kc$. Take $\mathtt{s}\in\mathtt{S}$
such that $|s-\mathtt{s}|<\eps_{1}$. By Lemma \ref{lem:transf_op_dependence_on_s}
and (\ref{eq:FIN1}), on $\mathscr{E}_{j}$ we have
\begin{align*}
\opnorm{\Lc_{s,\rho_{j}}^{\ell}} & \leq\opnorm{\Lc_{s,\rho_{j}}^{\ell}-\Lc_{\mathtt{s,\rho_{j}}}^{\ell}}+\opnorm{\Lc_{\mathtt{s,\rho_{j}}}^{\ell}}\\
 & <\eps_{1}J_{\Gamma}(\ell+1)[(2N_{\Gamma}-1)C_{\Kc}^{B_{\Gamma}}]^{\ell+1}+(1-\eps)=1-\frac{\eps}{2},
\end{align*}
which completes the proof.
\end{proof}

\section{Norm estimation of the limit operators\label{sec:Norm-estimation-of}}

The goal of this section is to prove the norm estimate of the limit
operators $\Tc_{\ell,s}$---Proposition \ref{prop:norm_Tks}---,
which was used to establish Theorem \ref{thm:sc_deterministic}. By
definition, these operators are finite sums of the form $\sum_{\gamma}T_{\gamma}\otimes\rho_{\Gamma}(\gamma)\in B(\Hc\otimes\ell^{2}(\Gamma))$,
where $\rho_{\Gamma}$ is the left regular representation of a free
group $\Gamma$ and $\Hc$ is a Hilbert space. Let $\TT$ be any operator
of this form. Here are the ingredients of our proof: the first and
key one is a bound due to Buchholz of $\norm{\TT}_{\Hc\otimes\ell^{2}(\Gamma)}$
in terms of the norms of certain matrices $\Rc_{j}$ with operator
coefficients, determined by the $(T_{\gamma})_{\gamma\in\Gamma}$.
This result is presented in Section \ref{subsec:Buchholz-inequality}.
Then, in Section \ref{subsec:Auxiliary-results} we prove a lemma
allowing us to bound $\opnorm{\Rc_{j}}$ by its ``Hilbert-Schmidt
norm'', and we estimate the size of the $(T_{\gamma})_{\gamma\in\Gamma}$
corresponding to the limit operators. We combine these results in
Section \ref{subsec:proof_bound_Tks} to prove Proposition \ref{prop:norm_Tks}.

\subsection{Buchholz inequality \label{subsec:Buchholz-inequality}}

In this section, $\Gamma$ is a free group of rank $N$ and we use
the following notation: Let us fix a basis $\Bc=\{\gamma_{1},\ldots,\gamma_{N}\}$
of $\Gamma$ and work with the word length $\wlength{\cdot}$ on $\Gamma$
with respect to $\Bc\cup\Bc\inv$. As before, $\Gamma_{k}$ is the
set of $\gamma\in\Gamma$ with $\wlength{\gamma}=k$. Recall that
$B(\Hc)$ is the group of bounded linear operators on a Hilbert space
$\Hc$. For any subset $\Lambda$ of $\Gamma$, we denote by $L^{2}(\Lambda,\Hc)$
the Hilbert space of $L^{2}$-maps $\Lambda\to\Hc$ with respect to
the counting measure on $\Lambda$. 

We start with some motivation: It was shown in \cite{haagerup_AnExample_1979}
that\emph{ }for any integer $m\geq1$ and any $\alpha\in L^{2}(\Gamma_{m},\CC)$
one has
\begin{equation}
\norm{\sum_{w\in\Gamma_{m}}\alpha_{w}\rho_{\Gamma}(w)}_{\ell^{2}(\Gamma)}\leq(m+1)\norm{\alpha}_{L^{2}}.\label{eq:Haagerup_ineq}
\end{equation}
This is commonly known as \emph{Haagerup's }inequality. It has various
applications and generalizations---see \cite{delaSalle_StrongHaagerup_2009}
for details---. What we need here is a version of (\ref{eq:Haagerup_ineq})
where the $\alpha_{w}$ are operators. This was established by Buchholz
in \cite{buchholz_norm_1999}.

Consider a map $R:\Gamma\to B(\Hc),w\mapsto R_{w}$. For any nonnegative
integers $m,n$ we define the map $\widetilde{\Rc}(m,n):L^{2}(\Gamma_{m},\Hc)\to L^{2}(\Gamma_{n},\Hc)$
by
\begin{equation}
[\widetilde{\Rc}(m,n)f](w_{n})=\sum_{w_{m}\in\Gamma_{m}}R_{w_{m}w_{n}}(f(w_{m})).\label{eq:def_Rmn}
\end{equation}
We set $\Rc(m,n)\eqdf\Ttt_{n}^{*}\widetilde{R}(m,n)\mathtt{\Ttt}_{m}$,
where $\mathtt{T}_{j}$ is the orthogonal projection $L^{2}(\Gamma,\Hc)\to L^{2}(\Gamma_{j},\Hc)$.
We state below only the inequality from \cite[Theorem 2.8]{buchholz_norm_1999}
that we need for our purposes.
\begin{thm}
\label{thm:Buchholz_bound}Let $\Gamma$ be a free group of finite
rank $N\geq2$. Consider a Hilbert space $\Hc$ and a map $\Gamma\to B(\Hc),\gamma\mapsto R_{\gamma}$
supported on $\Gamma_{\ell}$, for some integer $\ell\geq1$. Then
\[
\norm{\sum_{\gamma\in\Gamma_{\ell}}R_{\gamma}\otimes\rho_{\Gamma}(\gamma)}_{\Hc\otimes\ell^{2}(\Gamma)}\leq(\ell+1)\max_{0\leq j\leq\ell}\opnorm{\Rc(j,\ell-j)}.
\]
\end{thm}

\subsection{Auxiliary results\label{subsec:Auxiliary-results}}

Consider two Hilbert spaces $\Hc$ and $\Ic$. Let 
\[
\Hc^{\oplus j}\eqdf L^{2}(\{1,2,\ldots,j\},\Hc).
\]
For any linear map $T:\Hc^{\oplus m}\to\Ic^{\oplus n}$, we define
$T_{i,j}$ as $\mathtt{Q}_{i}T\mathtt{R}_{j}^{*}$, where $\mathtt{Q}_{i}$
and $\mathtt{R}_{j}$ are respectively the orthogonal projections
to the $i$-th and $j$-th factors of $\Hc^{\oplus m}$ and $\Ic^{\oplus n}$. 

The next lemma is a simple extension of a well-known inequality for
complex matrices.
\begin{lem}
\label{lem:opnorms_vs_HSnorm}For any bounded operator $T:\Hc^{\oplus m}\to\Ic^{\oplus n}$
we have
\[
\opnorm T\leq\left(\sum_{i=1}^{n}\opnorm{\sum_{j=1}^{m}T_{i,j}T_{i,j}^{*}}\right)^{\frac{1}{2}}.
\]
\end{lem}

\begin{proof}
We denote $\mathtt{Q}_{i}T$ by $T_{i}$. For all $v\in\Hc^{\oplus m}$
we have

\[
\norm{Tv}=\left(\sum_{i=1}^{n}\norm{T_{i}v}^{2}\right)^{\frac{1}{2}}\leq\left(\sum_{i=1}^{n}\opnorm{T_{i}}^{2}\right)^{\frac{1}{2}}\norm v.
\]
Thus, $\norm T\leq\left(\sum_{i}\opnorm{T_{i}}^{2}\right)^{\frac{1}{2}}$.
To conclude, note that $\opnorm{T_{i}}^{2}=\opnorm{T_{i}T_{i}^{*}}$
and
\[
T_{i}T_{i}^{*}=\sum_{j=1}^{m}T_{i,j}T_{i,j}^{*}.
\]
\end{proof}
We also need an estimate on the norm of the coefficients $M_{w,s}$---defined
in (\ref{eq:def_Mws})---of the limit operators.
\begin{lem}
\label{lem:bound_Mws}Let $\Gamma$ be a Fuchsian Schottky group.
There is $C_{\Gamma}>1$ with the next property: For any $w\in\Gamma_{\ell+1}$
with $\ell\geq1$ and all $s\in\CC$ with $\re s>0$ we have
\[
\opnorm{M_{w,s}}\llG e^{\pi|\im s|}C_{\Gamma}^{\re s}\len w^{\re s}.
\]
\end{lem}

\begin{proof}
Let $a=S(w)$ and $b=E(w)$, so that $\widetilde{M}_{w,s}:\Hc(D_{a})\to\Hc(D_{b})$.
Note that $\opnorm{M_{w,s}}=\opnorm{\widetilde{M}_{w,s}}$\footnote{Both operators have the same norm by the submultiplicativity of $\opnorm{\cdot}$
since $M_{w,s}=\Ptt_{b}^{*}\widetilde{M}_{w,s}\Ptt_{a}$ and $\widetilde{M}_{w,s}=\Ptt_{b}M_{w,s}\Ptt_{a}^{*}$.}. We work with the latter operator. Consider any $z_{b}\in D_{b}$
and a complex number $s=\sigma+it$ with $\sigma>0$. From the definition
of $\backspace w'(z_{b})^{s}$ we have

\[
|\backspace w'(z_{b})^{s}|=e^{-\alpha t}|\backspace w'(z_{b})|^{\sigma}
\]
for some $\alpha\in(-\pi,\pi)$. 

Now we compute the norm of $\widetilde{M}_{w,s}\psi_{a}$, for any
$\psi_{a}\in\Hc(D_{a})$, using the change of variable $z_{w}=\backspace wz_{b}$:
\begin{align*}
\norm{\widetilde{M}_{w,s}\psi_{a}}_{\Hc(D_{b})}^{2} & =\int_{D_{b}}|\backspace w'(z_{b})^{s}\psi_{a}(\backspace wz_{b})|^{2}\dd z_{b}\\
 & =\int_{D_{w}}\frac{|\backspace w'(z_{b})^{s}|^{2}}{|\backspace w'(z_{b})|^{2}}|\psi_{a}(z_{w})|^{2}\dd z_{w}\\
 & \leq e^{2\pi|t|}\int_{D_{w}}|\backspace w'(z_{b})|^{2\sigma-2}|\psi_{a}(z_{w})|^{2}\dd z_{w}.
\end{align*}
Since $\sigma>0$, applying (\ref{eq:upsilon-deriv-eq}) and Lemma
\ref{cor:norm_restriction_Bergman} we get

\begin{align*}
\norm{\widetilde{M}_{w,s}\psi_{a}}_{\Hc(D_{b})}^{2} & \leq e^{2\pi|t|}C_{\Gamma}^{2\sigma-2}\len w^{2\sigma-2}\norm{\psi_{a}}_{\Hc(D_{w})}^{2}\\
 & \llG e^{2\pi|t|}C_{\Gamma}^{2\sigma}\len w^{2\sigma}\norm{\psi_{a}}_{\Hc(D_{a})}^{2},
\end{align*}
from where our claim follows. 
\end{proof}
Let us write the limit operators $\Tc_{\ell,s}$ as a sum indexed
by $\Gamma_{\ell}$ instead of $\Gamma_{\ell+1}$:

\begin{align}
\Tc_{\ell,s} & =\sum_{w\in\Gamma_{\ell+1}}M_{w,s}\otimes\rho_{\Gamma}(\backspace w\inv)\nonumber \\
 & =\sum_{w\in\Gamma_{\ell}}\left(\sum_{\substack{b\in\Ac\\
w\inv\to b
}
}M_{w\inv b,s}\right)\otimes\rho_{\Gamma}(w)\nonumber \\
 & =\sum_{w\in\Gamma_{\ell}}T_{w,s}\otimes\rho_{\Gamma}(w),\label{eq:Limit_op}
\end{align}
where
\begin{equation}
T_{w,s}\eqdf\sum_{\substack{b\in\Ac\\
w\inv\to b
}
}M_{w\inv b,s}.\label{eq:def_Txs}
\end{equation}
These verify essentially the same bound as the $M_{w,s}$. Here are
the details:

\begin{lem}
\label{lem:bound_Tws} Let $\Gamma$ be a Fuchsian Schottky group.
There is a constant $C_{\Gamma}>1$ such that, for any $w\in\Gamma_{\ell}$
with $\ell\geq1$ and all $s\in\CC$ with positive real part we have
\[
\opnorm{T_{w,s}}\llG e^{\pi|\im s|}C_{\Gamma}^{\re s}\len w{}^{\re s}.
\]
\end{lem}

\begin{proof}
Let us fix $w\in\Gamma_{\ell}$. Note that the operators $M_{w\inv b,s}$
with $b\in\Ac$ such that $w\inv\to b$ have pairwise orthogonal images,
so 

\begin{equation}
\opnorm{T_{w,s}}\leq\sqrt{2N_{\Gamma}-1}\max_{\substack{\substack{b}
\in\Ac\\
w\inv\to b
}
}\opnorm{M_{w\inv b,s}}.\label{eq:AAA1}
\end{equation}
The concatenation (\ref{eq:upsilon-concatenation}) and mirror (\ref{eq:upsilon-mirror-eq})
estimates imply that $\len{w\inv b}\llG\len w$. Thus, for all $s\in\CC$
with $\re s>0$, Lemma \ref{lem:bound_Mws} yields
\[
\opnorm{M_{w\inv b,s}}\llG e^{\pi|\im s|}C_{\Gamma}^{\re s}\len w^{\re s},
\]
which combined with (\ref{eq:AAA1}) proves the bound on $\opnorm{T_{w,s}}$.
\end{proof}

\subsection{The proof of the norm bound of the limit operators \label{subsec:proof_bound_Tks}}
\begin{proof}
[Proof of Proposition \ref{prop:norm_Tks}]In all the proof we fix
$\ell\geq1,\eps>0$ and $s\in\CC$ with $\re s\geq\frac{\gexp}{2}+\eps$.
Consider the following map $R_{\bullet}:\Gamma\to B(\Hc(\Dne)):$
\[
R_{w}=\begin{cases}
T_{w,s} & \text{if }\wlength w=\ell,\\
0 & \text{otherwise.}
\end{cases}
\]
Recall that
\[
\Tc_{\ell,s}=\sum_{w\in\Gamma_{\ell}}R_{w}\otimes\rho_{\Gamma}(w)
\]
by (\ref{eq:Limit_op}). Since $R_{\bullet}$ is supported on $\Gamma_{\ell}$,
by Theorem \ref{thm:Buchholz_bound} we have\emph{
\begin{equation}
\opnorm{\Tc_{\ell,s}}\leq(\ell+1)\max_{0\leq j\leq\ell}\opnorm{\Rc(j,\ell-j)},\label{eq:Tks_Buchholz}
\end{equation}
}with $\Rc(j,\ell-j):L^{2}(\Gamma_{j},\Hc(\Dne))\to L^{2}(\Gamma_{\ell-j},\Hc(\Dne))$
as in (\ref{eq:def_Rmn}). We bound the norm of $\Rc(j,\ell-j)$ using
Lemma \ref{lem:opnorms_vs_HSnorm}, the triangle inequality for $\opnorm{\cdot}$
and the identity\footnote{Which holds for any bounded operator $T$ on a Hilbert space.}
$\opnorm{T^{*}T}=\opnorm T^{2}$ as follows:

\begin{align*}
\opnorm{\Rc(j,\ell-j)} & \leq\left(\sum_{w_{\ell-j}\in\Gamma_{\ell-j}}\sum_{\substack{w_{j}\in\Gamma_{j}\\
w_{j}\to w_{\ell-j}
}
}\opnorm{R_{w_{j}w_{\ell-j}}}^{2}\right)^{\frac{1}{2}}\\
 & =\left(\sum_{w\in\Gamma_{\ell}}\opnorm{T_{w,s}}^{2}\right)^{\frac{1}{2}}.
\end{align*}
By (\ref{eq:upsilon-exp-bound}), $\len w<1$ for any $w\in\Gamma_{\ell_{0}}$
provided $\ell_{0}\ggG1$. Assuming this holds for our initial $\ell$,
we have $\len w^{2\re s}\leq\len w^{\gexp+2\eps}$ for any $w\in\Gamma_{\ell}$.
From the bound for $\opnorm{T_{w,s}}$ of Lemma \ref{lem:bound_Tws},
and Lemma \ref{lem:exp_sums_of_boundary_lengths} we get

\begin{align}
\opnorm{\Rc(j,\ell-j)} & \llG e^{\pi|\im s|}C_{\Gamma}^{\re s}\left(\sum_{w_{\ell}\in\Gamma_{\ell}}\len{w_{\ell}}^{2\re s}\right)^{\frac{1}{2}}\nonumber \\
 & \leq e^{\pi|\im s|}C_{\Gamma}^{\re s}\left(\sum_{w_{\ell}\in\Gamma_{\ell}}\len{w_{\ell}}^{\gexp+2\eps}\right)^{\frac{1}{2}}\nonumber \\
 & \leq e^{\pi|\im s|}C_{\Gamma}^{\re s}\tau_{\Gamma}^{\ell\eps}.\label{eq:Tks_fin}
\end{align}
The claimed upper bound for $\opnorm{\Tc_{\ell,s}}$ follows from
(\ref{eq:Tks_Buchholz}) and (\ref{eq:Tks_fin}).
\end{proof}

\section{Proof of the spectral gap for random Schottky surfaces \label{sec:Proof-sp_gap_Schottky}}

In the situation of Theorem \ref{thm:sp_gap_random_Schottky}, our
strategy to prove that the random cover $X_{n}$ has no new resonances
in $\Kc$ is to show that, with high probability, some power $\Lc_{s,\rho_{n}^{0}}^{\ell}$
of the relevant random transfer operator has norm $<1$ for all $s\in\Kc$.
This is shown by feeding into Theorem \ref{thm:sc_deterministic}
a deep result on random matrices of Bordenave--Collins that we recall
in Section \ref{subsec:Background-on-random}. Having this, we complete
the proof of our spectral gap for random Schottky surfaces in Section
\ref{subsec:Proof-sp_gap_Schottky}\@.

\subsection{Background on random matrices\label{subsec:Background-on-random}}

The result of Bordenave-Collins is a contribution to a general problem
on random matrices that we now explain intuitively: The broad goal
is to study the joint behavior of finite sequences of independent,
random $n\times n$ matrices $X_{n,1},\ldots,X_{n,N}$ in some fixed
model $\Mc$, as $n\to\infty$. It is known that for many important
models , the matrices $X_{n,1},\ldots,X_{n,N}$ jointly converge to
deterministic operators $X_{\infty,1},\ldots,X_{\infty,N}$ of some
Hilbert space in the sense that, for any noncommutative polynomial
$P$ in $2N$ variables, the operators $P_{n}$ a.s. ``tend'' to
$P_{\infty}$, where $P_{j}=P(X_{j,1},\ldots,X_{j,N},X_{j,1}^{*},\ldots,X_{j,N}^{*})$.
This series of results started with the breakthrough work \cite{haagerup-thorb_ANewApplication_2005}
of Haagerup and Thorbjørnsen, where they treat the case of the Gaussian
Unitary Ensemble. The analog results for the Gaussian Orthogonal and
Symplectic Ensembles where established by Schulz in \cite{schultz_non-commutative_2005};
\cite{capitane-donati_StrongAsymptotic_2007} and \cite{anderson_ConvergenceWigner_2013}
treat some Wigner matrices under certain conditions on the distribution,
while \cite{collins-male_StrongAsFHaar_2014} consider random unitary
matrices with Haar distribution. Bordenave and Collins handle in \cite{bordenave-collins_EvRandomLifts_2019}
the case relevant for random covers of Schottky surfaces, namely,
when $X_{n,1},\ldots,X_{n,N}$ are the random unitary operators on
$V_{n}^{0}$ coming from a sequence $\varphi_{n,1},\ldots,\varphi_{n,N}$
of independent random permutations of $[n]$ with uniform distribution.

Let us now turn the random matrix intuition into a rigorous statement
in terms of representation theory. Besides being better suited for
our purposes, this language allows us to describe the limiting operators
in a natural and straightforward way: Let $\Gamma$ be a free group
of rank $N$ and consider a basis $\Bc=\{\gamma_{1},\ldots,\gamma_{N}\}$
of it. A sequence $\varphi_{n,1},\ldots,\varphi_{n,N}$ as above is
the same thing as a uniformly random homomorphism $\phi_{n}:\Gamma\to S_{n}$
simply by setting $\varphi_{n,j}=\phi_{n}(\gamma_{j})$. Thus, the
goal is to understand the random unitary representations $(\rho_{n}^{0},V_{n}^{0})_{n\geq1}$
of $\Gamma$. Below we reformulate \cite[Theorem 3]{bordenave-collins_EvRandomLifts_2019}
in terms of these representations. Recall that $\rho_{\Gamma}$ denotes
the left regular representation of $\Gamma$.
\begin{thm}[Bordenave--Collins]
\label{thm:BC}Let $\Gamma$ be a free group of finite rank. The
sequence $(\rho_{n}^{0},V_{n}^{0})_{n\geq1}$ of random representations
a.a.s. strongly converges to $(\rho_{\Gamma},\ell^{2}(\Gamma))$.
\end{thm}

\subsection{Proof of Theorem \ref{thm:sp_gap_random_Schottky}\label{subsec:Proof-sp_gap_Schottky}}

We are finally ready to prove our main result.
\begin{proof}[Proof of Theorem \ref{thm:sp_gap_random_Schottky}]
Suppose we are given a Schottky surface $X=\Gamma\backslash\Hyp$
and $\mathcal{K}$ as in the statement of Theorem \ref{thm:sp_gap_random_Schottky}.

By plugging Theorem \ref{thm:BC} into Theorem \ref{thm:sc_deterministic}
we see that there is $\ell=\ell(\Kc,\Gamma)$ such that the probability
that
\begin{equation}
\opnorm{\Lc_{s,\rho_{n}^{0}}^{\ell}}<1\text{ for all }s\in\Kc\label{eq:final1}
\end{equation}
tends to 1 as $n\to\infty$. When (\ref{eq:final1}) holds for a deterministic
homomorphism $\phi:\Gamma\to S_{n}$ and a fixed $s_{0}\in\Kc$ ,
1 is not an eigenvalue of the corresponding $\Lc_{s_{0},\rho}$. Hence
by Proposition \ref{prop:relation_eigenvalues_transf_op}, $s_{0}$
is not a new resonance of $X_{\phi}$. In our probabilistic setting,
we conclude that a.a.s. there are no new resonances of the random
cover $X_{n}$ in $\Kc$.
\end{proof}
\bibliographystyle{amsalpha}
\bibliography{Random_Schottky_surfaces}

\end{document}